\documentclass{amsart}
\usepackage{amscd,amssymb,amsmath,amsthm}
\usepackage[dvips]{graphicx}
\usepackage[all]{xy}
\theoremstyle{plain}
\newtheorem{definition}{Definition}
\newtheorem{proposition}{Proposition}
\newtheorem{theorem}[proposition]{Theorem}

\newtheorem{lemma}[proposition]{Lemma}

\newtheorem*{proposition*}{Proposition}
\newtheorem*{theorem*}{Theorem}
\newtheorem*{corollary*}{Corollary}
\newtheorem*{lemma*}{Lemma}
\newtheorem*{remark*}{Remark}
\newtheorem*{example*}{Example}

\newcommand{\Z}{\mathbb{Z}}
\newcommand{\Q}{\mathbb{Q}}
\newcommand{\R}{\mathbb{R}}
\newcommand{\C}{\mathbb{C}}

\begin{document}

\title{Quantum Master Equation and Open Gromov-Witten Theory 2}
 
\author{Vito Iacovino}


\email{vito.iacovino@gmail.com}

\date{version: \today}


\begin{abstract}

We define the not-abelian Open-Gromov-Witten partition function.

\end{abstract}

\maketitle

\section{Introduction}

Let $X$ be a Calabi-Yau simplectic six-manifold and let $L$ be a Maslov zero Lagrangian submanifold of $X$.
 In \cite{MCH} to the pair $(X,L)$ we associated a vector space called Multi-Curve-Homology ($MCH$). 
Given $\beta \in H_2(X,L)$,  from the perturbed moduli space of (multi-)pseudo-holomorphic-curves of class $\beta$ and fixed Euler characteristic $\chi$ it is constructed the Gromov-Witten $MC$-cycle $Z_{\beta}^{ab}$ :
\begin{equation} \label{corrispondence-ab}
 \text{Moduli space of (multi-)curves of class  }\beta \text{ and Euler Char. } \chi \leadsto Z_{\beta, \chi}^{ab}.
 \end{equation}
The relations defining $MC$-cycles correspond directly to the conditions on the perturbation of the moduli space of (multi-)pseudo-holomorphic-curves defined in \cite{MCH}
(it is this direct relation that makes the construction of the $MC$-cycle (\ref{corrispondence-ab}) relatively simple).
The construction depends on varies choices such as a perturbation of the moduli spaces. In the spirit of the general philosophy of Lagrangian Floer Theory,  different choices yield isotopic $MC$-cycles with isotopy canonically determinate up to isotopy of isotopies.

The Gromov-Witten $MC$-cycles $Z_{\beta}^{ab}$ may be considered as the universal invariant associated to the moduli space of holomorphic (multi-)curves of a fixed Euler Characteristic, which we refer as the abelian case. Considering ramified covering of $L$, the result of \cite{MCH} may also be applied to define the universal invariant associated to the moduli space of holomorphic (multi-)curves of a fixed genus and number of boundary components, which we refer as the not-abelian case (any ramified covering $\tilde{L}$ of the zero section of $T^*L$ may  be considered as a Lagrangian sub-manifold of $X$ after identify $T^*L$ with a neighborhood of $L$ in $X$, when $\tilde{L}$ varies over the different classes of ramified coverings, the Gromov-Witten $MC$-cycle (\ref{corrispondence-ab}) associated to $\tilde{L}$ defines the not-abelian universal invariant).
Alternatively it is possible to define a more refined version of the Gromov-Witten $MC$-cycle (\ref{corrispondence-ab}) which is constructed directly from the moduli space of (multi-)pseudo-holomorphic curves with boundary mapped in $L$ of a fixed genus and number of boundary components. In this paper we develop this approach.

The definition of the $MCH$ of \cite{MCH} involves graphs equipped with decoration: each vertex is decorated by an integer number (its Euler characteristic) and a class in $H_2(X,L)$ (its degree).  We do not equip the set of half-edges attached to each vertex with a cycle ordering. In the not-abelian case the graph is equipped with a more refined decoration: a finite set (its components), to each component it is associated a genus, a degree and a set of vertices,  the set of half-edges attached to a vertex is equipped with a cyclic order. So, in contrast with the abelian case, we do not have a genus and a degree associated to each single vertex.
Not-abelian graphs arise directly in the moduli space of pseudo-holomorphic curves. 
Indeed, in \cite{MCH} we first constructed a not-abelian $MC$-cycle $Z_{\beta}^{not-ab}$ and successively we specialized it to the abelian case. Hence the construction of \cite{MCH} applies to obtain the correspondence
\begin{equation} \label{correspondence-not-ab}
 \text{Moduli space of (multi-)curves of class }\beta \text{ genus  }g \text{ boundary comp. }h \leadsto Z_{\beta,g,h}^{not-ab}.
\end{equation}


 $MCH$ is strictly related to the mathematical formulation of the \emph{Point-Splitting}-Perturbative-Chern-Simons of \cite{PSPCS1}, \cite{PSPCS2}.
Indeed, in the particular case of trivial homology class $\beta=0$, $MCH$ reduces  to   the Point-Splitting-Homology($PSH$) introduced in \cite{PSPCS1} .
For not trivial homological classes  
 $MCH$ reduces to  the Point-Splitting-Homology of  \cite{PSPCS2} after we forget the decoration associated to the degree of the components. 
 To any $MC$-cycle it is associated a partition function which satisfies the Quantum Master Equation in a way similar to  \cite{PSPCS1}  and \cite{PSPCS2} .
 The relation between $MCH$ and $PSH$ should be considered as the natural mathematical formulation of the relation between Open Topological String and Perturbative Chern-Simons of \cite{W}.

To define $MCH$ we need to fix an annulus data, as in the case of $PSH$. Up to isotopies, this data defines a (smooth) Euler Structure, that is, a nowhere vanishing vector field on $L$.  In the correspondence (\ref{correspondence-not-ab}) the annulus data arise from the moduli space of area zero annulus. 
We extend the definition normalized-$PSH$  of  \cite{PSPCS2} in the context of $MCH$ leading to normalized-$MCH$. 
In {Point-Splitting}-Perturbative-Chern-Simons to each link equipped with an isotopy class of frames it is associated a normalized-$PSH$-class (the Euler structure is required to be compatible with the frame). An analogous construction applies in normalized-$MCH$ where we need to consider framed link decorated with a set of components, genus and degrees.     
The results of \cite{PSPCS2} are easily translated in the context of $MCH$ showing that normalized-$MCH$ is (up to small isotopies) isomorphic to a version of the Skein, which we call $MC$-Skein ( as in \cite{PSPCS2}, we consider the Skein before quotienting by isotopies).  The isomorphism is compatible with isotopies.
As in  \cite{PSPCS2},  besides the frame of the link, we need to include the extra geometric data given by Euler Alignments.  As in \cite{PSPCS2}, considering a suitable completion of the $MC$-Skein we may obtain a strict isomorphism.

\section{Multi-Curve Chain Complex}

In  this section we define Multi-Curve Chain Complex associated to the following data: 
\begin{itemize}
\item A compact oriented three manifold $M$,
\item a finite-rank abelian group $\Gamma$, called \emph{topological charges},
\item an homorphism of abelian groups
$$ \partial : \Gamma \rightarrow H_1(M, \Z) $$
called boundary homomorphism.
\item an homomorphism of abelian groups 
$$ \omega : \Gamma \rightarrow \R $$
called \emph{symplectic area}.
\end{itemize}

\subsection{Decorated Graphs} \label{section-graphs}
A decorated graph $G$ consists in an array 
$$(Comp, (V_c, D_c, \beta_c ,g_c )_c, (H_v)_v, E)$$ 
where 
\begin{itemize}
\item  A finite set $Comp(G)$, called set of components of $G$;
\item To each $c \in Comp(G)$ are assigned 
\begin{itemize}
\item a finite set $V_c$, called  vertices of $c$;
\item a finite set $D_c$, called degenerate vertices of $c$;
\item a class $\beta_c \in \Gamma$, called charge of $c$;
\item a positive integer numbers $g_c \in \Z_{\geq 0}$, called  genus of $c$.  
\end{itemize}
Set
$$\beta(G) =  \sum_{c \in Comp(G)} \beta_c \in \Gamma , V(G) = \bigsqcup_{c \in Comp(G) } V_c , D(G)= \bigsqcup_{c \in Comp(G)} D_c ;$$ 
\item  To each $v \in V(G)$ is assigned a cyclic ordered finite set $H_v$. Set 
$$H(G)= \bigsqcup_{v \in V_G} H_v$$
 the set of half-edges of $G$.
\item  $E(G)$ is a partition of $H(G)$ in sets of cardinality one or two, called set of edges $G$. The sets of cardinality two are called internal edges $E^{in}(G)$, the sets of cardinality one are called external edges $E^{ex}(G)$;
\end{itemize}

For each $c \in Comp$ we assume 
$$ \beta_c \in \Gamma_{tors} \Rightarrow \beta_c=0.$$

We assume that  there are no components purely closed, that is
$$  V_c \sqcup D_c \neq \emptyset  \quad \forall c \in Comp.$$


For any $c \in Comp$ set
$$\chi_c= 2-2g_c - |V_c| - |D_c|.$$
A component $c$ is called unstable if $\beta_c = 0$ and $2 \chi_c - | H_c | \geq 0 $.
The graph $G$ is called stable if all its components are stable.

There is an obvious notion of isomorphism of decorated graphs. 
We denote by $ \mathfrak{G}$ the set of isomorphism classes of stable decorated graphs.

The Euler characteristic of $G$ is defined as
$$\chi(G)= \sum_{c \in Comp} \chi_c - |E^{in}(G)|  .$$

Fix  and a norm  $ \parallel \bullet \parallel$ on $\Gamma_{\R} := \Gamma \otimes  \R$. 
For each positive real number $C^{supp}>0$  denote by $ \mathfrak{G}(\beta, \chi, |E^{ex}|, C^{supp})$ the set up to isomorphism of stable decorated graphs $G$ with  $\beta(G)=\beta, \chi(G)= \chi, |E^{ex}(G)| =|E^{ex}|$ satisfying the condition
 \begin{equation} \label{support}
\parallel \beta_c \parallel \leq C^{supp}  \omega(\beta_c)  \quad \forall c \in Comp.
\end{equation}

It is easy to check that $\mathfrak{G}(\beta, \chi, E^{ex}, C^{supp}) $ is a \emph{finite} set.
In the following  we fix the constant $C^{supp}$ and we omit the dependence on  $ \parallel \bullet \parallel$ and $C^{supp} $ in the notation. 



In the next subsection, for each $e \in E^{in}(G) $ it is defined the graph $\delta_e G$.  The operation $\delta_e$ associated two different edges commute:
$$   \delta_{e_1} \circ \delta_{e_2} G =  \delta_{e_2} \circ \delta_{e_1}  G \text{   for each   }  e_1 , e_2 \in E^{in}(G) \sqcup D(G).$$
Given a set of edges $\{e_1,e_2,...,e_n\} \subset E^{in}(G) $ we denote by $G/\{ e_1,e_2,...,e_n \} $ the graph that we get applying all the $\delta_{e_i}$ to $G$: 
$$  G/\{e_1,e_2,...,e_n\} =  \delta_{e_1} \circ \delta_{e_2} \circ ...  \circ \delta_{e_n} (G) . $$

Set
$$ \Sigma_G ;=  G / \{ E^{in}(G )  \} .$$ 
We have $E^{in}(\Sigma_G)= \emptyset $ , $\beta(G)= \beta(\Sigma_G)$ and $E^{ex}(G)=E^{ex}(\Sigma_G)$. We may identify $\Sigma_G$ with a (not necessarily connected) surfaces with boundary marked points $E^{ex}(G)$ and internal marked points $D(G)$ which we consider as degenerated boundary components.

We set  $g(G):= g(\Sigma_G )$ the genus of $G$, $h(G) := |V(\Sigma_G)| + |D(G)|$ the number of boundary components of $G$.

\subsubsection{Operation $\delta_v$, $v \in V(G)$}

 Let $G \in \mathfrak{G}$ and $v \in V(G)$ with $H_v = \emptyset$. Let  $c \in Comp_0(G)$ such that $v \in V_{c}$. The graph $\delta_{v} G$ is
  defined discarding $v$ from $V_{c}$ and adding a new element to $D_c$.

 All the other data defining $G$ stay the same.

\subsubsection{Operation $\delta_e$,  $e \in E^{in}(G) $}

For $G \in \mathfrak{G}$  and $e = \{ h_1,h_2 \} \in E^{in}(G) $ the graph $\delta_e G$  is defined as follows:
\begin{itemize}
\item
Assume $h_1 \in H_{v_1}$ and $h_2 \in H_{v_2}$ for $v_1 , v_2 \in V(G)$ with $v_1 \neq v_2$.  Define ordered sets $I_1$ and $I_2$ such that $H_{v_1 }= \{ h_1 , I_1  \}$ and  $H_{v_2 }= \{ h_2 , I_2  \}$ as cyclic ordered sets. $V(\delta_e G)$ is defined by replacing in $V(G)$ the vertices $v_1$ and $v_2$   by a unique vertex $v_0$, with $ H_{v_0} =  \{ I_1 , I_2  \} .$

\begin{itemize}
\item If  $v_1 \in V_{c_1}$, $v_2 \in V_{c_2}$ for  $c_1 , c_2 \in Comp(G)$ with $c_1 \neq c_2$, $Comp(\delta_e G)$ is obtained replacing in $Comp(G)$ the components $c_1$ and $c_2$ with a unique component $c_0$. $V_{c_0}$ is obtained by $V_{c_1} \sqcup V_{c_2} $ replacing the vertices $v_1$ and $v_2$ with $v_0$. $E_{c_0}=E_{c_1} \sqcup E_{c_2}$, $  g_{c_0} = g_{c_1}+ g_{c_2}$.
\item If  $v_1,v_2 \in V_{c_0}$ for some  $c_0  \in Comp(G)$, set $Comp(\delta_e G) =Comp(G)$ with the genus $g_{c_0}$ increased by one and all the other data of $c_0$ remain the same. 
\end{itemize}
\item
Assume $h_1,h_2 \in H_{v_0}$ for $v_0 \in V(G)$. Write the cyclic order set $H_{v_0 }$ as $H_{v_0 }= \{ h_1 , I_1 , h_2, I_2 \}$ for some order sets $I_1$ and $I_2$.  $V(\delta_e G)$ is given by $V(G)$ replacing $v_0$ by two vertices $v_0', v_0''$ with
$ H_{v_0'} =  I_1 \text{    }  H_{v_0''} =  I_2 .$

Set $Comp(\delta_e G) = Comp(G)$. Let $c_0 \in Comp(G)$ such that $v_0 \in Comp_{c_0}$. $V_{c_0}$ in $\delta_e G$ is obtained replacing $v_0$ with $v_0', v_0''$,   and all the other data of $c_0$ remain the same. 
\end{itemize}

\subsubsection{Marked Graphs} \label{marking-section}
 
Define the set of  decorated graphs $\mathfrak{G}_l$  as the set of pairs $(G,m)$, where $G \in \mathfrak{G}$ and  
$$m= \{ E_0, E_1,..., E_l \}$$ 
with $E_0 \subset E_1 \subset ... \subset E_l \subset E(G)$ be an increasing finite sequence of subsets of $E(G)$ with $E^{ex}(G) \subseteq E_0$.
The definition the operator $\delta_e$  works on $\mathfrak{G}_l$ for $e \in E^{in}(G) \setminus E_l$.  

 
 Denote
 $$ \mathfrak{G}_* = \sqcup_l  \mathfrak{G}_l .$$

Let  $cut_{E_0}G$ be the graph obtained cutting the edges $E_0$. So $E^{in }(cut_{E_0}G)= E^{in}(G) \setminus E_0^{in}, H(cut_{E_0}G)= H(G)$.  
We consider the decorated graphs up to the equivalence relation generated by
$$ (G,m ) \sim ( cut_{E_0}G, cut_{E_0}m   ) $$ 
where 
$$cut_{E_0}m =   \{ E^{ex}(cut_{E_0}G), E_1,..., E_l \}.$$

\subsection{Multi-Curves-Cycles}

As in \cite{PSPCS1} and \cite{PSPCS2}, in  this paper we consider chains always up to triangulations  and reparametrizations. The definition of $MC$-cycle we use in this paper is defined as in \cite{PSPCS1} using the decorated graphs defined above instead of the decorated graphs used in  \cite{PSPCS1}.

Fix an Annulus Data $\underline{Z}^{\circledcirc}$ defined as in \cite{PSPCS1}.

Fix $\beta \in \Gamma$.  A Multi-Curve-Cycle $Z$ of degree $\beta$ with annulus data $\underline{Z}^{\circledcirc}$ consists on a collection of chains
\begin{equation} \label{collection-chains-definition}
( Z_{G,m} )_{(G,m) \in \mathfrak{G}_*(\beta)}. 
\end{equation}
where
$${Z}_{G,m} \in C_{|H(G)| + |m|} ( M^{H(G)} , \mathfrak{o}_{H(G)})^{\text{Aut}(G,m)}  .$$
such  that 
\begin{itemize}
\item  the transversality condition is satisfied;
\item  the forgetful compatibility condition is satisfied  with annulus data $\underline{Z}^{\circledcirc}$;
\item the cycle condition satisfied.
\end{itemize}

The transversality condition, the forget compatibility condition and the cycle condition are defined as in \cite{PSPCS1}.

We denote by $ \mathcal{Z}_{\beta}(M|  {Z}^{\circledcirc}) $ the vector space of $MC$-cycles of degree $\beta$ with annulus data $\underline{Z}^{\circledcirc}$.


\subsubsection{Isotopies} \label{isotopies-section}

Similarly we may consider isotopies of $MC$-cycles:

Let $\tilde{Z}^{\circledcirc}$ be an isotopy between annulus data ${Z}^{\circledcirc,0}$ and  ${Z}^{\circledcirc,1}$. Take $Z^0 \in \mathcal{Z}(M|  {Z}^{\circledcirc,0})$ and $Z^1 \in \mathcal{Z}(M|  {Z}^{\circledcirc,1})$. An isotopy  $\tilde{Z} \in \tilde{\mathcal{Z}}$ of $MC$-cycles  between $Z^0$ and $Z^1$ of degree $\beta$ consists in a collection of Borel-Moore chains
\begin{equation} \label{collection-chains-isotopy}
\tilde{Z}=  ( \tilde{Z}_{G,m} )_{(G,m) \in \mathfrak{G}_*(\beta)} ,
\end{equation}
with
$$\tilde{Z}_{G,m} \in C_{|H(G)| + |m|+1} ( M^{H(G)} \times \R  , \mathfrak{o}_{H(G)})^{\text{Aut}(G,m)}  .$$
We require that
\begin{itemize} 
\item there exists  $\lambda^{\dagger} >0$  big enough such that
$$ \tilde{Z}^{<-\lambda^{\dagger}} =  Z^0  \times  \R_{<-\lambda^{\dagger}}  , \quad \tilde{Z}^{>\lambda^{\dagger}} =   Z^1  \times \R_{>\lambda^{\dagger}}   ;$$

 \item 
the transversality condition is satisfied;


\item the forgetful compatibility property is satisfied.

\item the cycle condition is satisfied.


\end{itemize}

We denote the space of isotopies of $MC$-cycles of degree $\beta$ with $ \tilde{\mathcal{Z}}_{\beta}(M| \tilde{Z}^{\circledcirc}) $.

\subsubsection{Isotopy of Isotopies}

Isotopy of isotopies of $MC$-cycles are defined as in the case of $PS$-cycles of \cite{PSPCS1}, using the decorated graphs introduced in this paper.

For simplicity  of notations below we omit  the dependence on the annulus data. We fix an isotopy of isotopies of annulus data.

\begin{definition} 
Let $ Z^{0, 0},   Z^{1, 0}  ,    Z^{ 0, 1}  ,    Z^{ 1 , 1} $ be $MC$-cycles. Let $ Z^{0, \bullet}$ be an isotopy between $Z^{0, 0}$ and $  Z^{ 0, 1}$,  $ Z^{1, \bullet}$ isotopy between $Z^{1, 0}$ and $  Z^{ 1, 1}$, $ Z^{ \bullet,0}$ isotopy between $Z^{0, 0}$ and $  Z^{ 1, 0}$,  $ Z^{ \bullet,1}$ isotopy between $Z^{0, 1}$ and $  Z^{ 1, 1}$.

An isotopy of isotopies of $MC$-cycles between $ Z^{0, \bullet},   Z^{1, \bullet}  ,    Z^{ \bullet, 0}  ,    Z^{ \bullet , 1} $ consists in a collection 
of Borel-Moore chains
\begin{equation} \label{collection-chains-isotopy-isotopy}
\tilde{ \tilde{Z}}=  (\tilde{\tilde{Z}}_{G,m} )_{(G,m) \in \mathfrak{G}_*(\beta)},
\end{equation}
with
$$\tilde{\tilde{Z}}_{G,m} \in C_{|H(G)| + |m|+2} (M^{H(G)} \times \R_{\lambda} \times \R_{\lambda'}  , \mathfrak{o}_{H(G)})^{\text{Aut}(G,m)}  .$$
such that
\begin{itemize}


\item
There exists  $\lambda^{\dagger} \gg 0$ big enough such that:
\begin{itemize}
\item $\tilde{\tilde{Z}}^{\lambda' <-\lambda^{\dagger} } =  Z_{ G,m}^{0, \bullet} \times \R_{\lambda'<-\lambda^{\dagger}}$;
\item $\tilde{\tilde{Z}}^{\lambda' >\lambda^{\dagger} } =  Z_{ G,m}^{1, \bullet} \times \R_{\lambda'>\lambda^{\dagger}} $;
\item  $\tilde{\tilde{Z}} ^{\lambda <-\lambda^{\dagger} } =  Z_{ G, m}^{ \bullet, 0 } \times \R_{\lambda < -\lambda^{\dagger}}$;
\item $\tilde{\tilde{Z}}^{ \lambda >\lambda^{\dagger} } =  Z_{ G,m}^{ \bullet, 1} \times  \R_{\lambda >\lambda^{\dagger}}$  .
\end{itemize}

\item  the transversality condition is satisfied;
\item forgetful compatibility condition is satisfied;

\item the cycle condition is satisfied:


\end{itemize}

We denote by $  \tilde{\tilde{ \mathcal{Z}}}_{\beta}$ the vector space of isotopies of isotopies.

\end{definition}

\subsubsection{Extension Lemmas}
 
 The extension Lemmas of \cite{PSPCS1},  \cite{PSPCS2} may be easily adapted to $MC$-cycles. 
 
 The partial order on the graphs and the truncated space of $MC$-cycles $\mathcal{Z}_{\prec G}$ are defined as in \cite{PSPCS1} .




\begin{lemma} \label{extension-PS-lemma}

Let $G \in \mathfrak{G}$. For any $Z_{\prec G} \in \mathcal{Z}_{\prec G}$ there exists  $Z_{\preccurlyeq G}  \in \mathcal{Z}_{\preccurlyeq G}$ extending $Z_{\prec G}$. 

\end{lemma}

  \begin{lemma} \label{extension-PS-lemma-isotopy}
  
Let $G \in \mathfrak{G}$ with $H(G) \neq \emptyset$. 
 Let $Z_{\preccurlyeq G}^0, Z_{\preccurlyeq G}^1 \in  \mathcal{Z}_{\preccurlyeq G} $ and  let  $\tilde{Z} _{\prec G} \in  \tilde{ \mathcal{Z}}_{\prec G} $ be an isotopy between
 $Z_{\prec G}^0$ and $ Z_{\prec G}^1$. 
 There exists  
 $\tilde{Z} _{\preccurlyeq G} \in  \tilde{ \mathcal{Z}}_{\preccurlyeq G} $ isotopy between 
 $Z_{\preccurlyeq G}^0$ and $Z_{\preccurlyeq G}^1 $ extending $\tilde{Z} _{\prec G} $.     
 \end{lemma}

An application of these two Lemma yields to the following:: 

  \begin{lemma} \label{extension-isotopies-inductive}
 Let $\underline{Z}_{\prec G}^0 \in  \mathcal{Z}_{\prec G} $,  $\underline{Z}_{\prec G}^1 \in  \mathcal{Z}_{\prec G} $ and    $\tilde{Z} _{\prec G} \in  \tilde{ \mathcal{Z}}_{\prec G} $  isotopy between $\underline{Z}_{\prec G}^0$ and $\underline{Z}_{\prec G}^1$.
 Let $\underline{Z}_{\preccurlyeq G}^0    \in  \mathcal{Z}_{\preccurlyeq G} $ extending $\underline{Z}_{\prec G}^0 $.
 

There exist
$\underline{Z}_{\preccurlyeq G}^1 \in  \mathcal{Z}_{\preccurlyeq G} $ and 
 $\tilde{Z} _{\preccurlyeq G} \in  \tilde{ \mathcal{Z}}_{\preccurlyeq G} $  isotopy between $\underline{Z}_{\preccurlyeq G}^0$ and $\underline{Z}_{\preccurlyeq G}^1  $ such that
\begin{itemize} 
\item $\underline{Z}_{\preccurlyeq G}^1  $ extends  $\underline{Z}_{\prec G}^1  $;
\item  $\tilde{Z} _{\preccurlyeq G}  $ extends $\tilde{Z} _{\prec G}  $. 
 \end{itemize}
 \end{lemma}

We also have the analogous for isotopies of isotopies:

\begin{lemma} \label{extension-isotopies-isotopies}

Let  $ Z_{ \preccurlyeq G}^{0, 0},   Z_{ \preccurlyeq G}^{1, 0}  ,    Z_{ \preccurlyeq G}^{ 0, 1}  ,    Z_{ \preccurlyeq G}^{ 1 , 1} $  be $MC$-cycles . Let $ Z_{ \preccurlyeq G}^{0, \bullet}$ be an isotopy between $Z_{ \preccurlyeq G}^{0, 0}$ and $  Z_{ \preccurlyeq G}^{ 0, 1}$,  $ Z_{ \preccurlyeq G}^{ \bullet,0}$ be an isotopy between $Z_{ \preccurlyeq G}^{0, 0}$ and $  Z_{ \preccurlyeq G}^{ 1, 0}$,  $ Z_{ \preccurlyeq G}^{ \bullet,1}$ be an isotopy between $Z_{ \preccurlyeq G}^{0, 1}$ and $  Z_{ \preccurlyeq G}^{ 1, 1}$,
$ Z_{ \prec G}^{1, \bullet}$ be an isotopy between $Z_{ \prec G}^{1, 0}$ and $  Z_{ \prec G}^{ 1, 1}$.

Let $\tilde{\tilde{Z}} _{\prec G} \in  \tilde{\tilde{ \mathcal{Z}}}_{\prec G} $  be an isotopy of isotopies between  $Z_{ \prec G}^{ \bullet ,0 },  Z_{ \prec G}^{ \bullet ,1 }, Z_{ \prec G}^{0 ,\bullet} ,Z_{ \prec G}^{1 ,\bullet} $. 

There exists $Z_{ \preccurlyeq G}^{1, \bullet} \in \tilde{ \mathcal{Z}}_{\preccurlyeq G} $ extending $Z_{ \prec G}^{1, \bullet}$ and  $\tilde{\tilde{Z}}_{\preccurlyeq G} \in  \tilde{\tilde{ \mathcal{Z}}}_{\preccurlyeq G} $ isotopy of isotopies between  $Z_{ \preccurlyeq G}^{ \bullet ,0 },  Z_{ \preccurlyeq  G}^{ \bullet ,1 }, Z_{ \preccurlyeq  G}^{0 ,\bullet} ,Z_{ \preccurlyeq  G}^{1 ,\bullet} $ extending
$\tilde{\tilde{Z}} _{\prec G} $. 

 \end{lemma}

\subsection{Gromov-Witten $MC$-cycle}

Let  $X$ be a Calabi-Yau simplectic six-manifold and $L$ be a Maslov index zero lagrangian submanifold.
We assume $[L] =0 \in H_3(X, \Z)$. 
Fix  a four chain $K$ with $\partial K =L$.

We set 
$$\Gamma = H_2(X,L, \Z)$$

In \cite{MCH} we provide the general construction of $MC$-cycles from the moduli space of multi-curves. 
We define a pertubation of the moduli space  of multi-curves satisfying certain compatibility conditions. The pertubation needs to be small enough in order for the construction of the virtual fondamental class to work. Since these conditions can be achieved only on a finite set of graphs from the moduli space of pseudo-holomorphic   (multi-)curve is obtained a $MC$-cycle only  up to finite order. 

Fix $k \in \Z_{>0}$ and $\beta \in H_2(X,L, \Z)$. Let $  \mathfrak{G}^k$ be the set of graphs $G \in \mathfrak{G}(\beta)$ with $- \chi(G) + |E^{ex}(G)| \leq k$.  We denote by  $\mathcal{Z}_k$ the space of $MC$-cycles truncated on $  \mathfrak{G}^k$ (see also next subsection).

\begin{theorem} \label{main-theorem} (\cite{MCH})
Let $\beta \in H_2(X,L, \Z)$. 
From the moduli space of pseudoholomorphic multi-curves of homology class $\beta$  we may construct  an Annuls data ${Z}^{\circledcirc}$ and  a $MC$-cycle $Z_{\beta} \in \mathcal{Z}_{\beta, k} (L | {Z}^{\circledcirc})$ up to order $k$, which we call Gromov-Witten $MC$-cycle.

 $Z_{\beta,k}$ depends by varies choices, i.e., the  Kuranishi structure  of the moduli space of multi-curves and its perturbation.
Different choices yield isotopic $MC$-cycles, with isotopy canonical up to isotopy of isotopies.
\end{theorem}

The Gromov-Witten Annulus data ${Z}^{\circledcirc}$ depends only on the moduli space of area zero multi curves of  genus zero and two boundary components.

To each annulus data it is associated an Euler Structure (see \cite{PSPCS1}). 
To the four chain $K$ it is associated an homology class of Euler Structures 
 $ [U_K] \in \mathfrak{Eul}(M)^{\blacklozenge}$  (see \cite{Boundary-States}): 
 $$  U_K=  \partial  \hat{K},$$
 where  $\hat{K} $ is the four chain on the differential blow-up $\hat{X} $ of $X$ along $L$ (we assume transversality  between $L$ and $K$).

\begin{proposition}  (\cite{Boundary-States})  
The homology class of the Euler Structure associated to the Gromov-Witten Annulus data agrees with $[U_K] $   .
\end{proposition}



\subsubsection{Full $MC$-cycle}




A full $MC$-cycle can be constructed from the result of \cite{MCH} using a formal argument as follows.

Consider  a finite subset of decorated graphs  $\mathfrak{G}' \subset \mathfrak{G}$ satisfying with the conditions
\begin{equation} \label{condition-troncation}
(G,m) \in  \mathfrak{G}',  (G',m') \prec (G,m) \Rightarrow (G',m') \in  \mathfrak{G}' .
\end{equation}
 We may consider $MC$-cycles truncated to   $\mathfrak{G}'$. Observe that the cyclic condition and forget compatibility condition make sense in this situation. 


Now let  $\mathfrak{G}^0 \subset \mathfrak{G}^1 \subset ....$  be an increasing sequence of subsets of  $\mathfrak{G}$ satisfying condition  (\ref{condition-troncation}) with $\sqcup_k  \mathfrak{G}^k = \mathfrak{G}$. 
For example we may take $  \mathfrak{G}^k$ as the set of graphs $G \in \mathfrak{G}$ with $- \chi(G) + |E^{ex}(G)| \leq k$.  Denote by  $\mathcal{Z}_k$ the corresponding space of $MC$-cycles.



Lemma  \ref{extension-isotopies-inductive} yields the following

 \begin{lemma} \label{extension-isotopies-inductive2}
 Let $\underline{Z}_k^0 \in  \mathcal{Z}_k $,  $\underline{Z}_k^1 \in  \mathcal{Z}_k $ and    $\tilde{Z} _k \in  \tilde{ \mathcal{Z}}_k $  isotopy between $\underline{Z}_k^0$ and $\underline{Z}_k^1$.
 Let $\underline{Z}_{k+1}^0    \in  \mathcal{Z}_{k+1} $ extending $\underline{Z}_k^0 $.
 

There exist
$\underline{Z}_{k+1}^1 \in  \mathcal{Z}_{k+1} $ and 
 $\tilde{Z} _{k+1} \in  \tilde{ \mathcal{Z}}_{k+1} $  isotopy between $\underline{Z}_{k+1}^0$ and $\underline{Z}_{k+1}^1  $ such that
\begin{itemize} 
\item $\underline{Z}_{k+1}^1  $ extends  $\underline{Z}_{k}^1  $;
\item  $\tilde{Z} _{k+1}  $ extends $\tilde{Z} _{k}  $. 
 \end{itemize}
 \end{lemma}

Lemma \ref{extension-isotopies-isotopies} yields the following

\begin{lemma} \label{extension-isotopies-isotopies2}

Let  $ Z_{ k+1}^{0, 0},   Z_{ k+1}^{1, 0}  ,    Z_{ k+1}^{ 0, 1}  ,    Z_{ k+1}^{ 1 , 1} $  be $MC$-cycles . Let $ Z_{ k+1}^{0, \bullet}$ be an isotopy between $Z_{ k+1}^{0, 0}$ and $  Z_{ k+1}^{ 0, 1}$,  $ Z_{ k+1}^{ \bullet,0}$ be an isotopy between $Z_{ k+1}^{0, 0}$ and $  Z_{ k+1}^{ 1, 0}$,  $ Z_{ k+1}^{ \bullet,1}$ be an isotopy between $Z_{ k+1}^{0, 1}$ and $  Z_{ k+1}^{ 1, 1}$,
$ Z_{ k}^{1, \bullet}$ be an isotopy between $Z_{ k+1}^{1, 0}|_{k}$ and $  Z_{ k+1}^{ 1, 1}|_k$.

Let $\tilde{\tilde{Z}} _{k} \in  \tilde{\tilde{ \mathcal{Z}}}_{k} $  be an isotopy of isotopies between  $Z_{ k+1}^{ \bullet ,0 }|_k,  Z_{ k+1}^{ \bullet ,1 }|_k, Z_{ k+1}^{0 ,\bullet}|_k ,Z_{ k}^{1 ,\bullet} $. 

There exists $Z_{ k+1}^{1, \bullet} \in \tilde{ \mathcal{Z}}_{k+1} $ extending $Z_{ k}^{1, \bullet}$ and  $\tilde{\tilde{Z}}_{k+1} \in  \tilde{\tilde{ \mathcal{Z}}}_{k+1} $ isotopy of isotopies between  $Z_{ k+1}^{ \bullet ,0 },  Z_{ k+1}^{ \bullet ,1 }, Z_{ k+1}^{0 ,\bullet} ,Z_{ k+1}^{1 ,\bullet} $ extending
$\tilde{\tilde{Z}} _{k} $. 

 \end{lemma}

The results of  \cite{MCH}  tells us that for each $k$ we may define pertubations etc, to obtain an element  $Z_k \in  \mathcal{Z}_k  $ for each $k$. Moreover we may assume that the perturbations are small enough such that there exists  $\tilde{Z}_k$  isotopy between $Z_{k+1}|_k$ and $Z_k$  for each $k$.

\begin{lemma}  \label{inductive-extension}
We may define inductively a $MC$-cycle $\dot{Z}_k$ and isotopy $\dot{\tilde{Z}}_k$ such that
\begin{itemize}
\item $\dot{Z}_k$ extends $\dot{Z}_{k-1}$;
\item $\dot{\tilde{Z}}_k$ is an isotopy between $Z_k$ and $\dot{Z}_k$. 
\end{itemize}
\end{lemma}
\begin{proof}
Assume that $\dot{Z}_k$ is defined. Apply Lemma \ref{extension-isotopies-inductive2} to the composition of $\tilde{Z}_{k}$ and $\dot{\tilde{Z}}_k$, to obtain $\dot{Z}_{k+1}$ and an isotopy  $\dot{\tilde{Z}}_{k+1}$ between $Z_{k+1} $ and $\dot{Z}_{k+1}$.
\end{proof}

Now we choice two complex structure,  perturbation,etc and obtain  as before $Z_k^{0}$ ,  $\tilde{Z}_k^{0}$, $Z_k^{1}$  , $\tilde{Z}_k^{1}$,  for each $k$. Assume also we have taken the perturbations small enough such that from the moduli space of multi-curves we  obtain isotopies $\tilde{Z}_k$ between $Z_k^{0}$ and $Z_k^{1}$, and an isotopy of isotopies $\tilde{\tilde{Z}}_k$ 
between  $\tilde{Z}_{k+1}|_k$,  $\tilde{Z}_k^0$ , $\tilde{Z}_k^1$, $\tilde{Z}_k$.




\begin{lemma}
Apply Lemma \ref{inductive-extension} to obtain $\dot{Z}_k^0 , \dot{\tilde{Z}}_k^0 , \dot{Z}_k^1,\dot{\tilde{Z}}_k^1$.

We may define inductively a $MC$-cycle isotopy $\dot{\tilde{Z}}_k$ and isotopy of isotopy $\dot{\tilde{\tilde{Z}}}_k$ such that
\begin{itemize}
\item $\dot{\tilde{Z}}_k$ is an isotopy between $\dot{Z}_k^0$ and $\dot{Z}_k^1$ extends $\dot{\tilde{Z}}_{k-1}$ ;
\item $\dot{\tilde{\tilde{Z}}}_k$ is an isotopy of isotopy between $\tilde{Z}_k$ , $\dot{\tilde{Z}}_k^0$, $\dot{\tilde{Z}}_k^1$ , $\dot{\tilde{Z}}_k$. 
\end{itemize}
\end{lemma}
\begin{proof}
The argument is the same of Lemma \ref{inductive-extension}  using  Lemma \ref{extension-isotopies-isotopies2}. 
\end{proof}

\subsection{Nice Multi-Curve-Cycles}

We consider arrays 
\begin{equation} \label{array-decoration}
(Comp^{\sharp}, (V_c^{\sharp}, \beta_c^{\sharp})_c  )
\end{equation}
where
\begin{itemize}
\item   $Comp^{\sharp}$ is a finite set; 
\item for each $c \in Comp^{\sharp}$ : $V_c^{\sharp}$ is a finite set and $\beta_c^{\sharp} \in \Gamma$.
\end{itemize}
We set
$$\beta^{\sharp} :=  \sum_{c \in Comp} \beta_c^{\sharp} \in \Gamma , V^{\sharp} := \bigsqcup_{c \in Comp } V_c^{\sharp}  .$$

A decorated multi-loop is defined by an array
\begin{equation} \label{decorated-multi-loop}
\overset{\bigtriangleup}{\mathbf{w}}= (Comp^{\sharp}, (V_c^{\sharp}, \beta_c^{\sharp})_c , \mathbf{w} )
\end{equation}
where
$(Comp^{\sharp}, (V_c^{\sharp}, \beta_c^{\sharp})_c  )$ is an array as in  (\ref{array-decoration}) and 
$\mathbf{w}= (w_v)_{v \in V^{\sharp}}$ is a multi-loop such that
$$ \partial \beta_c^{\sharp} = \sum_{v \in V_v}  [ w_v] \quad \forall c \in Comp^{\sharp} . $$

Given an array as in (\ref{array-decoration}) we denote by 
$$\mathfrak{G}(Comp^{\sharp}, (V_c^{\sharp}, \beta_c^{\sharp})_c  )$$
 the set of decorated graphs  $G \in \mathfrak{G}$ equipped with  the following extra data
 \begin{itemize}
\item 
  an identification 
$$Comp^{\sharp}=Comp_{\vartriangle}(G)$$
 satisfying the condition 
$$\beta_c^{\sharp}= \beta_c(G)  \quad \forall c \in Comp^{\sharp};$$
\item for each $c \in Comp^{\sharp}$ an inclusion
$$ V^{\sharp}_c  \hookrightarrow V_c(G)   .$$
 \end{itemize}

 We have a  map of sets
 $$\mathfrak{G}(Comp^{\sharp}, (V_c^{\sharp}, \beta_c^{\sharp})_c  ) \rightarrow \mathfrak{G},$$
 forgetting the extra decoration .

\begin{definition} \label{limit-support-definition}

Let $G$ be decorated graph. 
Given sequence of chains $\underline{C}= ( {\underline{C}}^{(n)} )_n$ on $\mathfrak{L}_GM$  we say that the limit support of $\underline{C}$
is finite if there exists a finite set of decorated multi-loops
$\{ \overset{\bigtriangleup}{\mathbf{w}^j} \}_{j \in J}$  such that for each $\epsilon$ the following happen for  $n$ big enough;
 for each $(w_v)_v \in supp (C^{(n)})$ there exists  $j \in J$ and a lift of $G$ to $\mathfrak{G}(Comp^{j}, (V_c^{j}, \beta_c^{j})_c  )$ such that
\begin{itemize}
\item  for each $c \in Comp_0(G)$, $Diam(\sqcup_{v \in V_c} supp(w_v)) < \epsilon $;
\item  for each $v \in V^j$, $w_v$ is $\epsilon$-close   to $w^j_v$   in the $C^0$-topology;
\item   for each $v \in V_{\vartriangle}(G) \setminus V^j$, $Diam(supp(w_v)) < \epsilon $.
\end{itemize}
We define the limit support as
$$\lim supp(\underline{C}) = \{ \overset{\bigtriangleup}{\mathbf{w}^j} \}_{j \in J}$$
where $ \{ \overset{\bigtriangleup}{\mathbf{w}^j} \}_{j \in J}$ is the minimal set of decorated  multi-loops satisfying the above conditions.

\end{definition}

\begin{definition} 
 A nice annulus data $\underline{Z}^{\circledcirc, \diamond}$ consists in a sequence of annulus data $(\underline{Z}^{\circledcirc, (n)})_n$ such that 
 \begin{itemize}
\item $ \lim Supp(\underline{Z}_{G,m}^{\circledcirc, \diamond})_{G,m}= \emptyset $ for each $(G,m) \in Ann0$ ;
\item there exists  a sequence $(\tilde{\underline{Z}}^{\circledcirc, int, (n)})_n$, with $ \tilde{\underline{Z}}^{\circledcirc, int, (n)}$  is an isotopy of annulus data between $ \tilde{\underline{Z}}^{\circledcirc,  (n-1)} $ and $ \tilde{\underline{Z}}^{\circledcirc,  (n)}$, such that
$ \lim Supp( pr_*(\tilde{\underline{Z}}^{\circledcirc, int}) =  \emptyset.$
\end{itemize} 
\end{definition}

  \begin{definition}
 Fix  a  nice annulus data $\underline{Z}^{\circledcirc, \diamond}$.  
  A nice $MC$-cycle  $Z^{\diamond } $  consists in a sequence of $MC$-cycles $( Z^{(n)} )_n$ such that 
\begin{itemize}
\item $Z^{(n)} \in \mathcal{Z}(M| \underline{Z}^{\circledcirc, (n)})$ for each $n$;
\item   $\lim Supp(Z_{G,m})$ is finite for each $(G,m) \in \mathfrak{G}_*$;
\item  there exists a sequence $(\tilde{Z}^{\circledcirc, int, (n)})_n$, with  $\tilde{Z}^{int, (n)}  \in \tilde{\mathcal{Z}}(M| \tilde{\underline{Z}}^{\circledcirc, int, (n)})$  isotopy of  $MC$-cycles   between $Z^{(n-1)} $ and $Z^{(n)}$, such that    
$$ \lim Supp( pr_*(\tilde{Z}_{G,m}^{int})) =   \lim Supp(Z_{G,m})  \qquad \forall (G,m).$$
\end{itemize} 
\end{definition}
We denote by $\mathcal{Z}(M| \underline{Z}^{\circledcirc, \circ})$ the space of nice $MC$-cycles and by $\mathcal{Z}(M| \overset{\bigtriangleup}{\mathbf{w}}, \underline{Z}^{\circledcirc, \circ})$ 
the space of nice $MC$-cycles whose limit support consists in one element  $\{\overset{\bigtriangleup}{\mathbf{w}} \}$.

Two nice $MC$-cycles  $Z^{0} $ and $Z^{1} $ are defined equivalent if there exists a sequence $( \tilde Z^{(n)} )_n$, 
with
$\tilde{Z}^{(n)}  \in \tilde{\mathcal{Z}}(M| {\underline{Z}}^{\circledcirc,  (n)})$  isotopy between $Z^{0,(n)} $ and $Z^{1, (n)} $, such that 
  $\lim Supp(pr_*(\tilde{Z}_{G,m}^{(n)}))_n$ is finite for each $(G,m)$  .


\begin{definition}
The Multi-Curve Homology  $MCH(M| \underline{Z}^{\circledcirc, \diamond})$ is defined as the vector space of equivalence classes of $MC$-cycles. 
\end{definition}

\subsubsection{Isotopies}

In order to define isotopies of nice $MC$-cycles we define the one-parameter family version of the limit support  of definition \ref{limit-support-definition} as follows.

Consider families of decorated multi loops $ \mathfrak{L}_GM \times \R$ of the type:
\begin{equation} \label{basic-family-loops}
(\{ (\overset{\bigtriangleup}{\mathbf{w}^\lambda},\lambda) \}_{ a \leq  \lambda \leq b})
\end{equation}
where $\overset{\bigtriangleup}{\mathbf{w}^\lambda}$ is a piecewise smooth one parameter family of decorated multi-loops, $a$ is a real number or  $ - \infty$,   $b$ is a real number or  $ + \infty$. We assume that $\mathbf{w}^\lambda$ is independent on $\lambda$ for $\lambda \gg 0$ or $\lambda \ll 0$,  if $b = \infty$  or $a = - \infty$ respectively.

\begin{definition}  \label{limit-support-definition-iso}
Let $G$ be decorated graph. 
Given a sequence of chains  $\tilde{\underline{C}}= ( {\tilde{\underline{C}}}^{(n)} )_n$  on $\mathfrak{L}_GM \times \R$  we say that the limit support  of $\underline{C}$
$\lim \tilde{supp}(\tilde{\underline{C}}) $ of $\tilde{\underline{C}}$
is finite if there exists a set  of decorated multi-loops
$\{ (\overset{\bigtriangleup}{\mathbf{w}^j}, \lambda^j ) \}_{j \in J}$  which is the union of a finite numbers of families as (\ref{basic-family-loops}) such that for each $\epsilon$ the following happen for  $n$ big enough: for each $((w_v)_v,t ) \in supp (\tilde{\underline{C}}^{(n)})$ there exists  $j \in J$ and a lift of $G$ to $\mathfrak{G}(Comp^{j}, (V_c^{j}, \beta_c^{j})_c  )$ such that
\begin{itemize}
\item  for each $c \in Comp_0$, $Diam(\sqcup_{v \in V_c} supp(w_v)) < \epsilon $;
\item  for each $v \in V^j$, $w_v$ is $\epsilon$-close   to $w^j_v$   in the $C^0$-topology  and $|t- t^j| < \epsilon$ ;
\item   for each $v \in V_{\vartriangle} \setminus V^j$, $Diam(supp(w_v)) < \epsilon $.
\end{itemize}
 We define the limit support as
$$\lim \tilde{supp}(\tilde{\underline{C}}) = \{ (\overset{\bigtriangleup}{\mathbf{w}^j}, \lambda^j ) \}_{j \in J},$$
where $ \{ (\overset{\bigtriangleup}{\mathbf{w}^j}, \lambda^j ) \}_{j \in J}$ is the minimal set satisfying the above conditions.

\end{definition}

\begin{definition}
Let $\tilde{\underline{Z}}^{\circledcirc, \diamond}$ be an isotopy of nice annulus data  between nice annulus data  $\underline{Z}^{0, \circledcirc, \diamond}$ and $\underline{Z}^{1,\circledcirc, \diamond}$.

For  $Z^{0,\diamond} \in \mathcal{Z}(M|   Z^{\circledcirc, 0 ,\diamond}    )$ and  $Z^{1 , \diamond} \in \mathcal{Z}(M|   Z^{\circledcirc, 1 ,\diamond}    )$ 
an isotopy  of nice $MC$-cycle between   $Z^{0, \diamond} $ and $Z^{1, \diamond} $  with annullus data $\tilde{\underline{Z}}^{\circledcirc, \diamond}$ consists in a sequence $\tilde{Z}^{\diamond }= ( \tilde Z^{(n)} )_n$, where $\tilde{Z}^{(n)} \in \tilde{\mathcal{Z}}(M| \tilde{\underline{Z}}^{\circledcirc, (n)})$   isotopy of $MC$-cycles between $Z^{0, (n)}$ and $Z^{1,(n)}$, with
$\lim \tilde{Supp} (\tilde Z_{G,m})$ is  finite  for each $(G,m)$.
\end{definition}

We denote by $ \tilde{\mathcal{Z}}(M| \tilde{\underline{Z}}^{\circledcirc, \diamond})$ the vector space of isotopies of nice $MC$-cycle  with nice annullus data $\tilde{\underline{Z}}^{\circledcirc, \diamond}$.

\subsubsection{Euler Structures}

In \cite{PSPCS1} we have proved a relation between Annulus datas and Euler Structures.  In order to do this it is usuful to consider an alternative formulation of the notion of Euler Structure. We consider pairs $U^{\blacktriangle}=(U_0 , U_1)$ where
\begin{itemize}
\item
 $U_0$ is a three chain on $M \times M$ close in the $C^0$ topology to $Diag_M$ and transversal to $Diag_M$;
\item $U_1$ is a one-chain on $M$ with
$$\partial U_1 = U_0 \cap Diag_M.$$
Here we consider $U_0 \cap Diag_M$ as a chain on $Diag_M \cong M$.  
\end{itemize}

There is an obvious notion of isotopy of these objects. Denote by $\mathfrak{Eul}(M)^{\blacktriangle}$  the set of isotopy classes of the $U^{\blacktriangle}$ .  The group $H_1(M, \Z)$ acts  on $\mathfrak{Eul}(M)^{\blacktriangle}$ by addiction on $U_1$.
With this action, $\mathfrak{Eul}(M)^{\blacktriangle}$ is a torsor on $H_1(M, \Z)$.  
It is easy to see that 
$  \mathfrak{Eul}(M) \cong \mathfrak{Eul}(M)^{\blacktriangle}$, where   $\mathfrak{Eul}(M)$ is the space of homology classes of the usual Euler Structures.

To an annulus data $Z^{\circledcirc}$ in \cite{PSPCS1} we associated an Euler Structure $U^{\blacktriangle}$,  canonically defined up to a set of contractible choices.  

The transfer map of $PSH$ can be extended to $MCH$ straightforwardly. 
So the same argument of  \cite{PSPCS1} show that the group  $MCH(M |   U^{\blacktriangle} ) $ is well defined. 



\subsubsection{From MC-cycles to nice MC-Cycles}

 \begin{lemma}
Let $Z^{\circledcirc}$ be an annulus data and let  $Z^{\circledcirc, \diamond}$ be a nice annulus data with $Z^{\circledcirc , \diamond, (0)}=Z^{\circledcirc}$. For each $MC$-cycle $Z^{} \in \mathcal{Z}(M | Z^{\circledcirc }  )$  there exists a nice $MC$-cycle $Z^{ \diamond} \in \mathcal{Z}(M|   Z^{\circledcirc, \diamond}    )$ with $Z^{ \diamond, (0)}=Z^{}$. 
  $Z^{ \diamond}$ is canonically determinate up to isotopies in the following sense.
 
 Let $\tilde{Z}^{\circledcirc}$ be an  isotopy between annulus data  $Z^{\circledcirc, 0}$  and  $Z^{\circledcirc, 1 }$. 
Let $\tilde{Z}^{\circledcirc, \diamond}$ be an isotopy between  $Z^{\circledcirc, 0 , \diamond }$  and  $Z^{\circledcirc, 1 , \diamond}$, with 
$Z^{\circledcirc , 0, (0)}=Z^{\circledcirc,0}$   and $Z^{\circledcirc ,  1, (0)}=Z^{\circledcirc,1}$,
 $\tilde{Z}^{\circledcirc,  (0)} =  \tilde{Z}^{\circledcirc}$.
 For each  $\tilde{Z} \in \tilde{\mathcal{Z}}(M|   \tilde{Z}^{\circledcirc}    )$ isotopy between  $Z^0$ and $Z^1$, there exists $\tilde{Z}^{\diamond} \in \tilde{\mathcal{Z}}(M|   \tilde{Z}^{\circledcirc , \diamond}    )$ isotopy between $Z^{ \diamond, 0}$ and $ Z^{ \diamond,1}$ with  $\tilde{Z}^{  (0)}=\tilde{Z}^{}$. 
   
  
 \end{lemma}

\begin{proof}

As usual we proceed by induction on  graphs. Assume  we have defined ${\underline{Z}}_{\prec G}^{(n)}, \tilde{\underline{Z}}_{\prec G}^{int,(n)}$ for each $n$.

Using induction on $n$, assume we have defined $\underline{Z}_{\preccurlyeq G}^{(n)}$. Apply Lemma  \ref{extension-isotopies-inductive} 
to $\underline{Z}_{\preccurlyeq G}^{(n)}, \tilde{ \underline{Z}}^{int, (n)}_{\prec G}$ and obtain $\underline{Z}_{\preccurlyeq G}^{(n+1)}, \tilde{ \underline{Z}}^{int, (n)}_{\preccurlyeq G}$. We claim that we may take the chains such that $\lim supp (Z^{\diamond}_{\preccurlyeq G})$ is finite.

By induction we may assume that  $\lim Supp \underline{Z}_{\prec G}$ consists of embedded links.

In the case  $H(G) \neq \emptyset$
 the proofs of Lemmas \ref{extension-PS-lemma} and  \ref{extension-PS-lemma-isotopy} tell us that the chains may be taken such that 
$$\lim supp (Z^{\diamond}_{\prec G})= \lim supp (Z^{\diamond}_{\preccurlyeq G}) , \lim \tilde{supp} (\tilde{Z}^{int}_{\prec G})= \lim \tilde{supp} (\tilde{Z}^{int }_{\preccurlyeq G}).$$

 Consider now  the case $H(G)= \emptyset$.  $\underline{Z}_G^{(n+1)}$ is isotopic to $\underline{Z}_G^{(n)} + \sum_{G', e'} \delta_{e'}\tilde{\underline{Z}}_{G'}^{int, (n)}  $.   For $n$ big enough we have 
$ supp(  \delta_{e'}\tilde{\underline{Z}}_{G'}^{int, (n)}   )$ is close in the $C^0$-topology to $\lim Supp \underline{Z}_{\prec G}$.
(Here we use the assumption that  $\lim Supp \underline{Z}_{\prec G}$ consists of embedded links.)
It follows that  for $n_0 \gg 0$  we may inductively define $ \underline{Z}_G^{(n)}$ such that
$supp(\underline{Z}_G^{(n)}) =   supp(\underline{Z}_G^{(n_0)}) \sqcup  \lim Supp \underline{Z}_{\preccurlyeq G} $,
 for each $n >n_0$.

Now consider  an isotopy 
 $\tilde{Z}^{} $ between $MC$-cycles  $Z^{0}$ and $Z^{1}$ as in the Lemma. 
 We want to define a sequence of isotopies $\tilde{Z}^{(n)}$  between  $Z^{ 0 , (n) }, Z^{ 1 , (n)}$ and a sequence of isotopies of isotopies 
$\tilde{\tilde{Z}}^{int, (n)}$ between   $\tilde{Z}^{(n-1)}, \tilde{Z}^{ int, 0 , (n) },\tilde{Z}^{int, 1 , (n)}, \tilde{Z}^{(n)}$.

Assume  we have defined $(\tilde{\underline{Z}}_{\prec G}^{(n)})_n$ and $ (\tilde{\tilde{\underline{Z}}}_{\prec G}^{int,(n)})_n$. We want to define their extension  $(\tilde{\underline{Z}}_{\preccurlyeq G}^{(n)})_n$ and $ (\tilde{\tilde{\underline{Z}}}_{\preccurlyeq G}^{int,(n)})_n$. We proceed by induction on $n$. Assume we have defined $\underline{Z}_{\preccurlyeq G}^{(n)}$. Apply Lemma \ref{extension-isotopies-isotopies} to  $ \tilde{\tilde{\underline{Z}}}_{\prec G}^{int,(n+1)}, \tilde{Z}_{\preccurlyeq G}^{(n)}, \tilde{Z}_{\preccurlyeq G}^{ int, 0 , (n+1) },\tilde{Z}_{\preccurlyeq G}^{int, 1 , (n+1)}, \tilde{\underline{Z}}_{\prec G}^{(n+1)}$ to obtain $\tilde{Z}_{\preccurlyeq G}^{(n+1)}, \tilde{\tilde{\underline{Z}}}_{\preccurlyeq G}^{int,(n+1)}$ . The chains defined in the proof of Lemma  \ref{extension-isotopies-isotopies} may be taken such that $\lim \tilde{Supp} (\tilde{\underline{Z}})$ is finite.



\end{proof}

\subsection{Forgetting the degree: Point Splitting Perturbative Chern Simons} \label{PCS-section}

In \cite{PSPCS1}, \cite{PSPCS2} we defined mathematically the Point Splitting Perturbative Chern Simons. In this subsection we show that  the relation between the $PSH$ defined in  \cite{PSPCS2} and the $MCH$ defined in this paper.

\subsubsection{Point-Splitting-Homology}

Let us first recall the definition of Point-Splitting-Homology. As for $MCH$,  the definition of $PSH$ is based on a set of decorated graphs, which we denote as $\mathfrak{G}^{\dagger}$ in order to distinguish it from the graphs used for $MCH$.

A decorated graph  $G \in \mathfrak{G}^{\dagger}$ is defined by the data 
\begin{equation} \label{array-point-splitting}
    ( \chi_{\vartriangle} , d_{\vartriangle} , V_{\vartriangle}, Comp_0 , (V_c,d_c,g_c)_c , (H_v )_v  , E )
\end{equation} 
where 
\begin{itemize}
\item $\chi_{\vartriangle}  \in \Z_{\leq 0}$;  
\item $d_{\vartriangle} \in  \Z_{\geq 0}$;
\item $ V_{\vartriangle}$ is a  finite set;
\item $Comp_0$ is a finite set;
\item for each $c \in Comp_0 $, $V_c$ is a finite set and $d_c, g_c \in \Z_{\geq 0}$;

Set 
$$V(G)= V_{\vartriangle} \sqcup \bigsqcup_c V_c .$$ 
\item for each $v \in V(G)$,  $H_v$ is a cyclically ordered finite set; 

Set 
$$H(G)= \bigsqcup_{v \in V(G)} H_v .$$
\item $E$ is a partition of $H$ in subset of cardinality two or one.  The sets of cardinality two are called internal edges $E^{in}(G)$, the sets of cardinality one are called external edges $E^{ex}(G)$;
\end{itemize}
For each $c \in Comp_0$, set 
$$\chi_c= 2- 2g_c - |V_c|  - d_c.$$
We say that $G$ is stable if
$$ \chi_c - \frac{1}{2} |H_c| <0 \quad \forall c \in Comp_0.$$
Denote by $ \mathfrak{G}$ the set of stable decorated graphs, up to isomorphism.



Set 
$$ \chi(G) = \chi_{\vartriangle}  + \sum_{c \in Comp_0} \chi_c , \quad  d(G)= d_{\vartriangle} + \sum_{c \in Comp_0} d_c  .$$



Denote by $\mathfrak{G}(\chi)$ the set of stable graphs $G$ with $\chi(G)= \chi$, up to isomorphism. 


Using the graphs $ \mathfrak{G}^{\dagger}$ in \cite{PSPCS2} it is defined the vector space of $PS$-cycles $\mathcal{Z}$.
The definition is parallel to the definition to the definition of $MC$-cycle provided in this paper 
 replacing the graphs  $ \mathfrak{G}$ with  $ \mathfrak{G}^{\dagger}$. Using $\mathcal{Z}$ we may  defined the space of  nice $PS$-cycles and the $PS$-homology as usual.

In the following we will need a shifted version of the $PSH$ defined above. For an positive integer number $N^{\ast}$, we denote by
$\mathfrak{G}^{\dagger}[ N^{\ast}]$  the set of decorated graphs defined using the same array (\ref{array-point-splitting})  where we replace the condition  $\chi_{\vartriangle} \leq 0$ with the condition  $\chi_{\vartriangle} \leq N^{\ast}$.  The definition of shifted-$PS$-homology $PSH(M|U)^{N^{\ast}}$ is as above using $\mathfrak{G}^{\dagger}[ N^{\ast}]$ instead of  $\mathfrak{G}^{\dagger}$.



\subsubsection{ $MCH \rightarrow PSH$}

From a $MC$-cycle we may define a $PS$-cycle as follows.

Given $\beta \in \Gamma$ we have a  map of sets  
\begin{equation} \label{graphs-MC-PS}
\mathfrak{G}(\beta) \rightarrow \mathfrak{G}^{\dagger}[ N_{\beta} ]  ,
\end{equation}
for $N_{\beta}$ integer big enough depending on $\beta$. Explicitly $N_{\beta}$ may be    defined as follows. 
For each $\underline{\beta} \in \Gamma/\Gamma_{tors}$ with $ ||\underline{\beta}|| < C^{supp} \omega(\underline{\beta})$, set
\begin{equation} \label{shift-beta}
N_{\underline{\beta}} = \max \{ N |  \exists (\underline{\beta}_1,  . . .  , \underline{\beta}_N),   \underline{\beta}_i \in \Gamma/\Gamma_{tors} ,  ||\underline{\beta}_i|| < C^{supp} \omega(\underline{\beta}_i) \text{  for   } i=\{1,2,...,  N \} ,    \sum_i \underline{\beta}_i = \underline{\beta} \}.   
\end{equation}
For $\beta \in \Gamma$ we set $N_{\beta}= N_{\underline{\beta}}$,  where $\underline{\beta}$ is the image of $\beta $ in $\underline{\beta} \in \Gamma/\Gamma_{tors}$.

The map (\ref{graphs-MC-PS}) is defined as follows. 
To the decorated graph 
$$G= (Comp, (g_c, \beta_c, D_c , V_c)_c, (H_v)_v ,E ) \in \mathfrak{G} (\beta)$$ 
corresponds  the decorated graph 
 $$ ( \chi_{\vartriangle} , d_{\vartriangle} , V_{\vartriangle}, Comp_0 , (V_c,d_c,g_c)_c , (H_v )_v  , E )$$
where
$$  V_{\vartriangle}= \bigsqcup_{c \in Comp_{\vartriangle} } V_c ,$$
$$\chi_{\vartriangle }= \sum_{c \in  Comp_{\vartriangle}} \chi_c, $$
$$d_{\vartriangle }= \sum_{c \in  Comp_{\vartriangle}}  |D_c|, $$
$$ d_c= |D_c|  \quad \text{  for   } c \in Comp_0,$$
$(H_v)_v$ and  $E$ are the same.




From the support condition (\ref{support}) it follows that  $\chi_{\vartriangle} \leq N_{\beta}$.

\subsection{Factorization Property}

The Factorization property in the not abelian context may be addressed analogously to the abelian case of \cite{QME}. It can be considered as an extension of the factorization property arising in $PSPCS$ of \cite{PSPCS1}.



To state the factorization property we introduce Factorization $MC$-cycles as follows. We considered decorated graphs 
$$\mathfrak{G}^{!}(\beta_1,\beta_2)=  \mathfrak{G}(\beta_1) \times \mathfrak{G}(\beta_2).$$
Denote by $G^1 \oplus G^2$ the disjoint union of the decorated graphs $G^1,G^2 \in \mathfrak{G}$.

Let $\mathfrak{G}^{!}_*$ be the set of pairs $((G^1,G^2) , m)$ with $(G^1,G^2) \in \mathfrak{G}^{!}$ and $m$ is a marking  for $G^1 \oplus G^2$   as in  Subsection \ref{marking-section} . Note that 
define $m$ is equivalent to define  a marking $m^1$ for $G^1$ and a marking $m^2$ for $G^2$. We write $m= m^1 \oplus m^2$.

Fix an  annulus data $\underline{Z}^{\circledcirc}$.
A factorization $MC$-cycle $Z^!_{\beta_1,\beta_2}$ 
consists on a collection of chains
$$( Z^!_{(G^1,G^2) , m)} )_{((G^1,G^2),m) \in \mathfrak{G}_*^!(\beta_1,\beta_2)} $$
with
$$ Z^!_{(G^1,G^2) , m)}  \in C_{|H(G^1)| + |H(G^2)|+ |m|} ( M^{H(G^1) \sqcup H(G^2)} , \mathfrak{o}_{H(G^1) \sqcup H(G^2)})^{\text{Aut}(G^1,G^2) , m)}  .$$
such that
\begin{itemize}
\item the transversality condition is satisfied; 
\item  the forgetful compatibility is satisfied  with annulus data $\underline{Z}^{\circledcirc}$ ;
\item the cycle condition is satisfied .
\end{itemize}

The transversality condition, the forgetful compatibility condition and the cycle condition for factorization $MC$-cycles are defined analogously to $MC$-cycles.

 We denote by $\mathcal{Z}^!$  the vector space of factorization $MC$-cycles.

Isotopies of  factorization $MC$-cycles can be defined analogously  modifying the definition of isotopies of  subsection ( \ref{isotopies-section}).





\subsubsection{Factorization Map}

Define the factorization map between the usual space of $MC$-cycles and the factorization $MC$-cycles:
\begin{equation} \label{factorization-map}
 \mathbf{fact}_{\beta_1,\beta_2} :  \mathcal{Z}_{\beta} \rightarrow  \mathcal{Z}^!_{\beta_1, \beta_2}
\end{equation}
\begin{equation} \label{factorization-formula-definition}
 \mathbf{fact}(Z)_ {(( G^1,G^2),m  ) }  := Z_{(G^1 \oplus G^2,  m) }. 
 \end{equation}

\begin{lemma}
For each $MC$-cycle $Z_{\beta}$,  $   \mathbf{fact}(Z)$ is a factorization $MC$-cycle.
\end{lemma}
\begin{proof} It is an easy check.
\end{proof}

The factorization map  for isotopies of $MC$-cycles 
\begin{equation} \label{factorization-map-iso}
 \mathbf{fact} : \tilde{\mathcal{Z}} \rightarrow  \tilde{\mathcal{Z}}^!.
\end{equation}
is defined from
\begin{equation} \label{factorization-formula-definition-iso}
 \mathbf{fact}(\tilde{Z})_ {(( G^1,G^2),m  ) }  := \tilde{Z}_{(G^1 \oplus G^2,  m) }. 
 \end{equation}


\subsubsection{Product}

We define the  product
\begin{equation} \label{product-factorization}
  \boxtimes : \mathcal{Z}_{\beta_1}  \times  \mathcal{Z}_{\beta_2} \rightarrow \mathcal{Z}^!_{\beta_1, \beta_2} ,
  \end{equation}
$$   (Z^1, Z^2) \rightarrow  Z^1\boxtimes Z^2$$
using the formula
\begin{equation} \label{cup-formula}
(Z^1\boxtimes Z^2)_{(G^1,G^2), m } := \sum_{0 \leq r \leq l}  Z^1_{(G^1, m^1_{[0,r]})} \times  Z^2_{(G^2, m^2_{[r,l]})}. 
\end{equation}
Here we have used the notation $m_{[a,b]} := \{ E_i \}_{a \leq i \leq b}$ when $m = \{ E_i \}_{0 \leq i \leq l}$ and $0 \leq a \leq b \leq l$.

\begin{lemma}
For each $MC$-cycles $Z^1, Z^2$,  $ Z^1\boxtimes Z^2$ is a factorization $MC$-cycle.
\end{lemma}
\begin{proof} It is an easy check.
\end{proof}

The definition of product is extended to isotopies
\begin{equation} \label{product-factorization-isotopies}
  \boxtimes : \tilde{\mathcal{Z}}_{\beta_1}  \times  \tilde{\mathcal{Z}}_{\beta_2} \rightarrow \tilde{\mathcal{Z}}^!_{\beta_1 + \beta_2} ,
  \end{equation}
$$   (\tilde{Z}^1, \tilde{Z}^2) \rightarrow  \tilde{Z}^1\boxtimes \tilde{Z}^2$$
from
\begin{equation} \label{cup-formula-iso}
(\tilde{Z}^1\boxtimes \tilde{Z}^2)_{(G^1,G^2), m } := \sum_{0 \leq r \leq l}  \tilde{Z}^1_{(G^1, m^1_{[0,r]})} \times_{\R}  \tilde{Z}^2_{(G^2, m^2_{[r,l]})}. 
\end{equation}

\subsubsection{Nice MC-cycles}

Nice factorization $MC$-cycles $\mathcal{Z}^{!, \diamond}$ may be defined as nice $MC$-cycles
using  factorization $MC$-cycles  instead of $MC$-cycles. The isotopies $\tilde{Z}^{int,(n)}$ used in the definition are isotopies of  factorization $MC$-cycles.

From $\mathcal{Z}^{!, \diamond}$ we may define the factorization $MC$-homology $MCH^!(M  | \underline{Z}^{\circledcirc, \diamond}  )$ taking the equivalence classes.


As in \cite{PSPCS1} we have the following easy Lemmas

 \begin{lemma}
 For each ${Z}^{ \diamond}  = (Z^{(n)})_n \in \mathcal{Z}^{!, \diamond}$ . $(\mathbf{fact}(Z^{(n)}))_n$ defines a nice factorization $MC$-cycles. We denote it as $ \mathbf{fact}(Z^{\diamond})$.
\end{lemma}

 \begin{lemma}
For each ${Z}^{ 1, \diamond}  = (Z^{1, (n)})_n , {Z}^{ 2, \diamond}  = (Z^{2, (n)})_n \in \mathcal{Z}^{!, \diamond}$  , $( Z^{1,(n)} \boxtimes Z^{2,(n)})_n$ defines a nice factorization $MC$-cycles. We denote it as $ Z^{1, \diamond} \boxtimes Z^{2, \diamond}$.
\end{lemma}

\begin{lemma}
The product  (\ref{cup-formula}) induces  a product in $   MCH(M| \underline{Z}^{\circledcirc, \diamond})$
\begin{equation} \label{product-nice-homology}
    MCH(M| \underline{Z}^{\circledcirc, \diamond}) \boxtimes MCH(M| \underline{Z}^{\circledcirc, \diamond}  ) \rightarrow MCH^!(M  | \underline{Z}^{\circledcirc, \diamond}  ).
\end{equation}

The product (\ref{product-nice-homology})  is commutative .
\end{lemma}
\begin{proof}
The proof is the same as in \cite{PSPCS1}.
\end{proof}

\subsubsection{Factorization Property}

\begin{definition}
We say that  the collection of $MCH$-classes $([Z^{\diamond}_{\beta}])_{\beta}$ , $ [Z^{\diamond}_{\beta}]  \in    MCH_{\beta}(M| \underline{Z}^{\circledcirc, \diamond})$,
 satisfies the factorization condition if 
\begin{equation} \label{factoriation-equation}
  \mathbf{fact}_{\beta_1, \beta_2}([Z^{\diamond}]) = [Z^{\diamond}_{\beta_1}] \boxtimes [Z^{\diamond}_{\beta_2}] 
  \end{equation}
  for each $\beta_1 , \beta_2  \in \Gamma $. 
  \end{definition}



\begin{proposition} (\cite{QME}) 
The collection of $MCH$-classes $([Z^{\diamond}_{\beta}])_{\beta}$  of Theorem  \ref{main-theorem}  may be taken such that  the factorization condition is satisfied. 
\end{proposition}

\section{Open Gromov-Witten Partition Function and Potential}

Let $M$ be a compact oriented three-manifold and let $\mathfrak{g}$ be a principal flat vector bundle over $M$.  Denote by $H^*(M, \mathfrak{g})$  the cohomology  of $M$ with coefficients in $\mathfrak{g}$ endowed with the odd symplectic form induced by the pairing
$$\langle \alpha \otimes X,  \alpha' \otimes X' \rangle =
 \int_M \alpha \wedge \alpha'  \langle  X,  X' \rangle_{\mathfrak{g}} .$$
In this section we associate to any  $MC$-cycle its partition function, a formal power series on the formal variable $\hbar$ and the variable $x \in (H^*(M, \mathfrak{g})[1])^*$.
The construction follows the same lines of the construction of \cite{PSPCS1} and \cite{PSPCS2}.

As in \cite{PSPCS1}, the partition function will depend not only on the choice of flat connection of the $\mathfrak{g}$-bundle, but also on the choice of a representation
$$ \rho :  \mathfrak{g} \rightarrow \mathfrak{gl}(n, \C) .$$
In the following we implicitly identify $\mathfrak{g}$ with a sub-Lie-algebra of $\mathfrak{gl}(n, \C)$ using $\rho$.

As in \cite{PSPCS1} we fix a finite dimensional vector subspace $\mathcal{A} < \Omega^*(M, \mathfrak{g})$ of the space of closed differential forms of $M$ with value in $\mathfrak{g}$ such that its map in cohomology of $M$ yields an isomorphism:
\begin{equation} \label{harmonic-isomorphism}
   \mathcal{A} \xrightarrow{\,\raisebox{-0.4ex}{$\sim$}\,} H^*(M, \mathfrak{g}) .
   \end{equation}
Equivalently we have  a linear map
$$ r_{\mathcal{A}}: H^*(M, \mathfrak{g}) \rightarrow \Omega^*(M, \mathfrak{g})$$
such that 
$$ r_{\mathcal{A}}([\alpha])= \alpha \quad \forall \alpha \in \mathcal{A} .$$

Applying  ( \ref{harmonic-isomorphism}) to the first factor of   $Identity \in H^*(M, \mathfrak{g})[1] \otimes (H^*(M, \mathfrak{g})[1])^*$ we obtain an element 
$$\mathfrak{x} \in \mathcal{A} \otimes \R[x] .$$

Let $G \in \mathfrak{G}$. 
For $h \in H(G)$ denote by $\pi_h : M^{H(G)} \rightarrow M $  the projection on the point $h$, set
$$ \mathfrak{g}_h = \pi_h^* (\mathfrak{g}) . $$

Define the bundle $\mathfrak{g}_{H(G)}$ over $M^{H(G)}$
$$ \mathfrak{g}_{H(G)} = \bigwedge (\bigoplus_{h \in H(G)} \mathfrak{g}_h) .$$

Denote by $\mathfrak{g}_{H(G)}^{\vee}$ the dual of the vector bundle $\mathfrak{g}_{H(G)} $.

The  trace map of \cite{PSPCS2} gives a map
\begin{equation} \label{multi-loop-fiber-map}
C_*(\mathfrak{L}_{G}(M) , \mathfrak{o}(H(G))^{\text{Aut}(G,m)} \rightarrow C_*(M^{H(G)},(\mathfrak{g}_{H(G)})^{\vee} )^{\text{Aut}(G,m)}.
\end{equation}





Fix a propagator $P \in \Omega_2(\text{Conf}_2(M),   \pi_1^*(\mathfrak{g}) \otimes \pi_2^*(\mathfrak{g})  )$ defined as in \cite{PSPCS1}.   $P$ depends on the choice of $\mathcal{A} < \Omega^*(M, \mathfrak{g})$ as above and a differential form representing the generator of $H^2(\partial \text{Conf}_2(M) , \Z )^-$.  Different choices leads to isotopic propagators (see \cite{PSPCS1}).

As in \cite{PSPCS1}, to a decorated graph $(G,m)$  we may associate a manifold with corners $Conf_{G,m}(M)$.

If  there exist
  $$e_i \in E^{in}, 1 \leq i \leq l, $$
  such that for each $k$ with $ 1 \leq k \leq l$ we have
  $$E_k= E_0 \sqcup \{e_1,e_2,  . . . ,e_k   \}, $$
to $(G,m)$  we associate a differential form
$$\underline{\Omega}_{G,m} \in  \R[x] \otimes \Omega^*(Conf_{G,m}(M)) \otimes  \mathfrak{g}_{H(G)}   $$
defined by
\begin{equation} \label{omega-graphs}
\underline{\Omega}_{G,m}   = (-1)^l \bigwedge_{e \in E(G) \setminus E_l} \pi_e^*(P_s) \wedge \bigwedge_{i} \pi_{e_i}^*(dP_s) \wedge \bigwedge_{e \in E^{ext}(G)}  \pi_e^*( \mathfrak{x} )  
\end{equation}
where $P_s$ is the super-propagator corresponding to $P$, i.e. , the image of $P$ under the map $ \pi_1^*(\mathfrak{g}) \otimes \pi_2^*(\mathfrak{g})   \rightarrow \bigwedge(\pi_1^*(\mathfrak{g}) \oplus \pi_2^*(\mathfrak{g})) .$

 Note that , using  (\ref{multi-loop-fiber-map}), we may define the pairing
 $ \langle   \underline{\Omega}_{G,m} ,   \underline{Z}_{G,m}  \rangle$.

 The partition function of the $MC$-cycles $Z \in \mathcal{Z}_{\beta}$ is defined as the formal power series
\begin{equation} \label{partition-definition}
    \mathfrak{P} (Z) :=  \sum_{(G,m) \in \mathfrak{G}_*(\beta)}  \hbar^{- \chi(G)}  N^{|D(G)|}  \langle   \underline{\Omega}_{G,m} ,   \underline{Z}_{G,m}  \rangle .
\end{equation}

Here,  as in \cite{PSPCS1}, 
$\mathfrak{P} (Z)$ defines a formal Laurent series $ \R((\hbar))[[x]]  $. Write
$$ \mathfrak{P} (Z) = \sum_{l,k} \hbar^l p_{k,l}(x) $$
where, for each $k \geq 0$ , $p_{k,l}(x)$ is an homogeneus polynomial of degree $k$. In our situation the stability condition on the graphs  tells us
$$ p_{k,l} \neq 0 \Rightarrow k + 3(l + N_{\beta}) \geq 0$$
where $N_{\beta}$ are the integer numbers defined in  (\ref{shift-beta}).




\subsubsection{Isotopies}
As in \cite{PSPCS1}, \cite{PSPCS2}, we may consider isotopies of $MC$-cycles.
 To an isotopy $ \tilde{Z} = (\tilde{Z}_{G,m})_{G,m} \in \tilde{\mathcal{Z}}_{\beta}$ of  $MC$-cycles it is associated the isotopy of partition functions defined by
\begin{equation} \label{partition-definition-isotopies}
    \mathfrak{P} (\tilde{Z}) :=   \sum_{(G,m) \in \mathfrak{G}_*(\beta)} \hbar^{- \chi(G)}  N^{d(G)} \langle   \tilde{\underline{\Omega}}_{G,m} ,   \tilde{\underline{Z}}_{G,m}  \rangle
\end{equation}

As in \cite{PSPCS1}, \cite{PSPCS2}, 
$ \mathfrak{P} (\tilde{Z})$  defines a formal Laurent series with coefficients piecewise smooth differential forms on $\R$. We may write
$$ \mathfrak{P} (\tilde{Z}) = \sum_{l,k} \hbar^l \tilde{p}_{k,l}(x) $$
where now $\tilde{p}_{k,l}(x)$ is an homogeneus polynomial of degree $k$ with coefficients in piecewise smooth differential forms on $\R$.  Now the stability condition on the graphs  tells us
$$ \tilde{p}_{k,l} \neq 0 \Rightarrow k + 3l + N_{\beta} \geq 0 . $$

We may write
$$ \mathfrak{P} (\tilde{Z})= \mathfrak{Q} (\tilde{Z}) + \mathfrak{K} (\tilde{Z}) d \lambda , $$
where $\mathfrak{Q} (\tilde{Z})$ and $\mathfrak{K} (\tilde{Z}) $ are formal Laurent series with coefficients piecewise smooth functions on $\R$.

We have , for generic value of $\lambda$,
$$\mathfrak{Q} (\tilde{Z})= \{ \lambda \mapsto \mathfrak{P} ({Z}^{\lambda}) \}.$$
where $ \underline{{Z}}^{\lambda}$  is  the $MC$-cycle defined by the restriction of $\tilde{Z}$ to $\lambda$.

\begin{lemma}
$ \mathfrak{Q} (\tilde{Z})$ is a continuos function on the variable $\lambda$. 
\end{lemma}
\begin{proof}
The proof  is the same of \cite{PSPCS1}.
\end{proof}




\begin{proposition} \label{QME-proposition}
$\mathfrak{P} (\tilde{Z})$ satisfies the $QME$:
\begin{equation} \label{master-equation}
    d_{\lambda} \mathfrak{P} (\tilde{Z}) + g_s \Delta \mathfrak{P} (\tilde{Z}) =0
\end{equation}
\end{proposition}
\begin{proof}
The proof  is the same of \cite{PSPCS1}.
\end{proof}

From  Theorem \ref{main-theorem} we conclude that from the moduli space of pseudo-holomorphic curves we obtain a partition function   $\mathfrak{P} ({Z}^{GW})$ defined up to isotopies.

\subsection{Bulk Deformations}

 Bulk deformations can be included in a similar way to \cite{QME}. We need to consider decorated graphs with internal punctures $\mathfrak{G}^{+}$ and using the corresponding  version of the $MC$-cycles. 


  A decorated graph     $G^+ \in \mathfrak{G}^+$ consists in an array 
\begin{equation} \label{array-bulk}
(Comp, (V_c, P_c , D_c, \beta_c ,g_c )_c, (H_v)_v, E)
\end{equation}
where 
\begin{itemize}
\item for each $c \in Comp$,  $P_c$ is a finite set, called internal punctures.
\end{itemize}
The other data of the array (\ref{array-bulk}) are as in $\mathfrak{G}$.
Set $P(G^+) = \sqcup_c P_c$.


The space of $MC$-cycles with bulk deformations $\mathcal{Z}^+$  is defined using  collections of chains      
$$( Z_{(G^+,m)})_{(G^+,m) \in \mathfrak{G}_*^+(\beta  ) }$$
 where 
 $$Z_{G^+,m} \in C_*( L^{H(G^+)}).$$
The transversality condition, the forgetful compatibility condition and the cycle condition are extended immediately in this context.

Define $\Omega_{G^+,m}$ using the same formula (\ref{omega-graphs}). To $MC$-cycle $Z= (Z_{G^+,m})_{G^+,m}$  with bulk deformations we associate its partition function
$$ \mathfrak{P} (Z) :=    \sum_{(G^+,m) \in \mathfrak{G}_*^+(\beta)} g_s^{- \chi(G)} \mathfrak{b}^{|P(G)|} N^{|D(G)|}  \langle   \Omega_{G^+,m} ,   Z_{G^+,m}  \rangle ,$$
where $\mathfrak{b}$ is a new formal variable associated to the number of internal punctures.

$\mathfrak{P} (Z)$ defines a formal Laurent series $ \R((\hbar))[[x,   \mathfrak{b} ]]  $. Write
$$ \mathfrak{P} (Z) = \sum_{l,k} g_s^l p_{k,l}(x,   \mathfrak{b}  ) $$
where $p_{k,l}(x,  \mathfrak{b}  )$ is an homogeneus polynomial of degree $k$ for each $k \geq 0$, when we set $deg(x)=1, deg(\mathfrak{b})=2$. The stability condition on the graphs  tells us
$$ p_{k,l} \neq 0 \Rightarrow k +  3 N_{\beta}  + 3l    \geq  0. $$

\subsubsection{Open Gromov-Witten Partition Function}

It is easy to adapt the construction of \cite{MCH} to obtain  from the moduli space of multi curves with bulk deformations the not abelian Gromov-Witten $MC$-cycle $Z^{not-ab}=(Z_{G^+,m})_{G^+,m} $ with bulk deformations (see \cite{QME} ).


Denote by $Z^{K,A}_{\beta}$ the not-abelian Open Gromov-Witten nice $MC$-cycle constructed using the four chain $K$ and the bulk deformation $A$. 
Set 
$$ \mathfrak{P} (\beta,  K, A)= \mathfrak{P} (Z^{K,A}_{\beta}) .$$

For any $r \in \Q$, from the construction it is immediate to obtain the identity
$$  \mathfrak{P} (\beta, K+ r A, A)  (g_s, \mathfrak{b}) =  \mathfrak{P} (\beta, K, A)  (g_s, \mathfrak{b} + r g_s ).$$
 This identity tells us that the bulk deformation may be considered as a deformation of the four-chain $K$.

\section{Normalized MCH}

In this section we introduce a less refined version of $MCH$ which is the analogous of  normalized $PSH$ defined in  \cite{PSPCS2}. 



A decorated graph $G$ consists of an array 
$$(\chi_{\vartriangle }, d_{\vartriangle}, Comp,  (V_c, \beta_c  )_{c \in Comp},  ( d_c, g_c )_{c \in Comp_0}, (H_v)_v, E)$$ 
where 
\begin{itemize}
\item $\chi_{\vartriangle}  , d_{\vartriangle} \in  \Z$;
\item   $Comp$ is a finite set;
\item for each $c \in Comp$ , $\beta_c \in \Gamma$,  $V_c$ is a finite set. 

Let $Comp_0$ be the set of $c \in Comp$ with $\beta_c=0$ and $Comp_{\vartriangle}$ be the set of $c \in Comp$ with $\beta_c \neq 0$. So $Comp= Comp_0 \sqcup Comp_{\vartriangle}$. 
\item for each $c \in Comp_0$,  $ d_c, g_c \in  \Z_{\geq 0}$;
Set
$$\beta(G) :=  \sum_{c \in Comp(G)} \beta_c \in \Gamma , V(G) := \sqcup_{c \in Comp(G) } V_c, V_{\vartriangle}(G) =  \sqcup_{c \in Comp_{\vartriangle} } V_c  ;$$ 
 \item  for each $v \in V(G)$, $H_v$ is a cyclically ordered set;
  Set $H(G)= \sqcup_{v \in V_G} H_v$.
\item  $E(G)$ is a partition of $H(G)$ in sets of cardinality one or two, called set of edges $G$. The sets of cardinality two are called internal edges $E^{in}(G)$, the sets of cardinality one are called external edges $E^{ex}(G)$;
\end{itemize}

We assume 
$$ \beta_c \in \Gamma_{tors} \Rightarrow \beta_c=0 ,$$
$$  d_{\vartriangle} \geq 0, \chi_{\vartriangle}  \leq  |Comp_{\vartriangle} |.$$


For each $c \in Comp_0$ set 
$$\chi_c= 2-2g_c - |V_c| -d_c$$
We say that a component $c$ is unstable if $\beta_c = 0$ and $2 \chi_c - | H_c | \geq 0 $.
We say that the graph $G$ is stable if each of its components is stable.

 
 Let $\mathfrak{G}^{ \blacktriangledown}$ be the set of graphs $G^{} \in \mathfrak{G}^{}$ with at least a component $c \in Comp_0(G^{}) $  with $V_{c} = \emptyset$.
 Let $\mathcal{Z}^{null}$ be the set of $MC$-cycles $Z$ with support  in  $\mathfrak{G}^{ \blacktriangledown}$.   

Set 
$$   MCH^{\intercal}(M| \underline{Z}^{\circledcirc}) = MCH(M|  \underline{Z}^{\circledcirc})/\mathcal{Z}^{null}. $$

Two $MC$-cycles are equivalent in  $MCH^{\intercal}(M| \mathbf{w}, \underline{Z}^{\circledcirc})$ if they are isotopic up to elements of $\mathcal{Z}^{null}$.


The  normalized-$MC$-Homology is equipped with a natural structure of module over the ring of formal power series $\Q[[\hbar, a]]$, where $\hbar$ and $a$ are formal variables.
 
The action of $(\hbar^{k_0} a^{d_0})_{k_0,d_0 \in \Z_{\geq 0} }$ on $PSH$ is inducted by the action on $\mathfrak{G}$ defined by 
$$ \kappa_{\vartriangle} (\hbar^{k_0} a^{d_0} G) = \kappa_{\vartriangle} (G) + k_0,  d_{\vartriangle} (\hbar^{k_0} a^{d_0} G) = d_{\vartriangle}(G) + d_0 .$$

\begin{proposition} \label{one-dimension}
Let $ \underline{Z}^{\circledcirc, \diamond} $ be a nice annulus data. 
When $ \mathbf{w}$ is an embedded link , $MCH^{\intercal}(M| (\mathbf{w} , G^{\sharp}), \underline{Z}^{\circledcirc , \diamond}  )$ is a rank one free module over $\Q[[\hbar, a]]$.
\end{proposition}

\subsection{$MCH \rightarrow normalized-MCH $}

 
For each $\beta \in \Gamma$ we have  a map of sets  
\begin{equation} \label{graphs-MC-normalized-MC}
\mathfrak{G}(\beta) \rightarrow \mathfrak{G}^{\dagger}( \beta)  .
\end{equation}
 
To the decorated graph 
$$G= (Comp, (g_c, \beta_c, D_c , V_c)_c, (H_v)_v ,E ) \in \mathfrak{G} (\beta)$$ 
corresponds  the decorated graph 
$$G^{\dagger}=(\chi_{\vartriangle }, d_{\vartriangle}, Comp,  (V_c, \beta_c  )_{c \in Comp},  ( d_c, g_c )_{c \in Comp_0}, (H_v)_v, E)  \in \mathfrak{G}^{\dagger} $$
defined by
$$\chi_{\vartriangle }= \sum_{c \in Comp_{\vartriangle}} \chi_c, d_{\vartriangle }= \sum_{c \in Comp_{\vartriangle}} |D_c|,$$
$$ d_c= |D_c| \quad \forall c \in Comp_0,$$
and the other data are the same.

\subsection{Coherent $MC$-Cycles}





Let $(\mathbf{w, \nu})$ be a framed link.
As in \cite{PSPCS2}, we say that an Euler Structure $U^{\blacktriangle}$ is compatible with $\nu$ if $U^{\blacktriangle}|_{\mathbf{w}}$ is constant when written in the trivialization of $T M |_{\mathbf{w}}  $ defined by $\nu$.




\begin{proposition} 

Let $\overset{\bigtriangleup}{\mathbf{w}}= (Comp^{\sharp}, (V_c^{\sharp}, \beta_c^{\sharp})_c , \mathbf{w} )$ be a decorated oriented embedded link as in (\ref{decorated-multi-loop}). 
Let $\nu$ be a frame of ${\mathbf{w}}$ and let $ U^{\blacktriangle}$ be
an Euler Structure compatible with $\nu$. 

 To the triple $( \overset{\bigtriangleup}{\mathbf{w}}, \nu, U^{\blacktriangle}  )$  it is associated canonically a normalized $MCH$-class 
\begin{equation} \label{coherent-class}
[Z^{( \overset{\bigtriangleup}{\mathbf{w}}, \nu, U^{\blacktriangle}  ) }] \in MCH^{\intercal}(M| \overset{\bigtriangleup}{\mathbf{w}},  U^{\blacktriangle} ).
\end{equation}

 Let $\overset{\bigtriangleup}{\tilde{\mathbf{w}}}$  be an  isotopy of decorated embedded oriented links between  $\overset{\bigtriangleup}{\mathbf{w}^0}, \overset{\bigtriangleup}{\mathbf{w}^1}$ . 
  Let $\nu^0$ and $ \nu^1$ be frames of    $\mathbf{w}^0, \mathbf{w}^1$ and let $U^{\blacktriangle,0},U^{\blacktriangle,1} $ be Euler Structures compatible with  $\nu^0, \nu^1$ respectively.
Let  $\tilde{\nu}$ be an isotopy between  $\nu^0$ and $ \nu^1$ and
let $\tilde{U}^{\blacktriangle}$ be an isotopy  between   $U^{\blacktriangle,0}$ and $U^{\blacktriangle,1} $ compatible with $\tilde{\nu}$. We have
\begin{equation} \label{transfer-coherent-normalized}
Transfer_{ \tilde{\mathbf{w}}, \tilde{U}^{\blacktriangle}}( [Z^{( \overset{\bigtriangleup}{\mathbf{w}^0}, \nu^0, U^{\blacktriangle}  ) }  ]  )   =   [Z^{( \overset{\bigtriangleup}{\mathbf{w}^1}, \nu^1, U^{\blacktriangle}  ) }  ]  .
\end{equation}

\end{proposition}

The result is deduced as the analogous result of  \cite{PSPCS1}  and \cite{PSPCS2} extending the definition of flat $PS$-cycles and flat approaching $PS$-cycles to $MC$-cycles as follows.

 Fix a finite set $V^{\sharp}$. In \cite{PSPCS2}, in order to define flat-$PS$-cycles   
associated to $V^{\sharp}$ we introduced the set of decorated graphs $\mathfrak{G}^{\dagger}(V^{\sharp})$. An element $G^{\dagger} \in \mathfrak{G}^{\dagger}$ is equipped with an inclusion map 
$V^{\sharp} \hookrightarrow V_{\vartriangle}(G^{\dagger}).$
Then we introduced the set of partitioned decorated $\mathfrak{G}^{\dagger,part}(V^{\sharp})$ which are pairs $(G, \mathcal{P})$ where $G \in \mathfrak{G}^{\dagger}(V^{\sharp})$ and  $\mathcal{P}$ is an admissible partition of
$$V^+(G) :=  (V(G) \setminus  V^{\sharp}) \sqcup \bigsqcup_{v^{\sharp} \in V^{\sharp}} H_{v^{\sharp}}.$$

The definition  of admissible partition is extended to the context of this paper as follows.

\begin{definition}
Let $G \in \mathfrak{G}(Comp^{\sharp}, (V_c^{\sharp}, \beta_c^{\sharp})_c  )$ . An admissible partition $\mathcal{P}$ of  $V^+(G)$ is a partition of $V^+(G)$ equipped with an array 
$$ ( \mathcal{P}^{v^{\sharp}} )_{v^{\sharp} \in V^{\sharp}}$$
 where each $\mathcal{P}^{v^{\sharp}}$ is a subset of $\mathcal{P}$ endowed with a cyclic order. We assume the following conditions:
\begin{itemize}
\item  the sets $ \{ \mathcal{P}^{v^{\sharp}} \}_{v^{\sharp} \in V^{\sharp}}$ are disjoint;
\item  for any $v^{\sharp} \in V^{\sharp}$ we have  
$$H_{v^{\sharp}} \subseteq  \bigsqcup_{P \in \mathcal{P}^{v^{\sharp}}} P$$ 
and the cyclic order of $H_{v^{\sharp}}$ is compatible with the cyclic order of $ \mathcal{P}^{v^{\sharp}}$;
\item any element of $V_{\vartriangle}$ belongs to a $P \in \mathcal{P}^{v^{\sharp}}$ for some $v^{\sharp} \in V^{\sharp}$;
\item for each $P \in \mathcal{P}^{v^{\sharp}}$ there exists at least an element of $V_{\vartriangle}$ or of $H_{v^{\sharp}}$ belonging to $P$;
\item    for each $ v^{\sharp} \in V^{\sharp}$ and  $P \in \mathcal{P}^{v^{\sharp}}$ we have $P \subset V_c \sqcup V_0$, where $c \in Comp$ is the component with $ v^{\sharp} \in V_c$.
\end{itemize}

We denote by $\mathfrak{G}^{part}$  the set of pairs $(G, \mathcal{P})$ where $G \in \mathfrak{G}(Comp^{\sharp}, (V_c^{\sharp}, \beta_c^{\sharp})_c  )$ and $ \mathcal{P}$ is an admissible partition of $V^+(G)$.
$\mathfrak{G}^{part}_*$ denotes the set of triples $(G, m,\mathcal{P})$ with $(G,m)  \in  \mathfrak{G}_* $ and $(G, \mathcal{P}) \in \mathfrak{G}^{part}$.

\end{definition}

We denote by $\mathcal{P}^0$ the set  
$$\mathcal{P}^0 =    \mathcal{P} \setminus  \bigsqcup_{v^{\sharp} \in V^{\sharp}} \mathcal{P}^{v^{\sharp}}. $$
From the conditions above it follows that an element $P \in  \mathcal{P}$ belongs to $\mathcal{P}^0$ if and only if $P \subset V_0(G)$.

We endow  the set of admissible partitions of $V^+(G)$ with the following partial order:  $\mathcal{P'} \preccurlyeq \mathcal{P}$  if  
\begin{itemize}
\item   for  each $v_1,v_2 \in V^+(G)$
 $$v_1 \sim_{\mathcal{P}'} v_2 \Rightarrow v_1 \sim_{\mathcal{P}} v_2 ;$$
\item the cyclic order of $ \mathcal{P}^{\prime, v^{\sharp}}$ is compatible with the  the cyclic order of $ \mathcal{P}^{ v^{\sharp}}$.
\end{itemize}

Using $\mathfrak{G}^{part}$ we may define flat $MC$-cycles and flat-approaching $MC$-cycles similarly to the definition of flat-$PS$-cycles and flat-approaching-$PS$-cycles  provided in \cite{PSPCS2}.
 The the existence and unicity results  of  flat-$PS$-cycles and flat-approaching $PS$-cycles given in \cite{PSPCS2}  may be easily translated to normalized-$MCH$.  
 
, 
We can also construct directly a flat-$MC$-cycle from a  flat $PS$-cycle and a flat-approaching-$MC$-cycle from a flat-approaching $PS$-cycle. The construction is based on the  map 
\begin{equation} \label{lift-partitions}
 \mathfrak{G}^{\dagger, part}(V_{\sharp}) \rightarrow  \mathfrak{G}^{part}(Comp^{\sharp}, (V_c^{\sharp}, \beta_c^{\sharp})_c ),
\end{equation}
which associate to the partitioned $PS$-graph $(G^{\dagger}, \mathcal{P}) \in \mathfrak{G}^{\dagger, part}(V^{\sharp})$ the  partitioned $MC$-graph $(G, \mathcal{P}) \in \mathfrak{G}^{part}(Comp^{\sharp}, (V_c^{\sharp}, \beta_c^{\sharp})_c )$ defined as follows:
$$V(G)=V(G^{\dagger}), H(G)=H(G^{\dagger}), E(G)=E(G^{\dagger}),$$
$$ \chi_{\vartriangle}(G)=  \chi_{\vartriangle}(G^{\dagger}) + |Comp^{\sharp}|,d_{\vartriangle}(G)=  d_{\vartriangle}(G^{\dagger}) ,$$
$$Comp_{\vartriangle}(G)= Comp^{\sharp} , Comp_0(G)= Comp_0(G^{\dagger}),$$
$$ V_c =  V_c^{\sharp} \sqcup \bigsqcup_{v^{\sharp} \in V_c^{\sharp} } \bigsqcup_{P \in  \mathcal{P}^{v^{\sharp}}} P \cap V_{\vartriangle}(G^{\dagger}) \quad \forall c \in Comp^{\sharp}.$$

Finally, the argument of  \cite{PSPCS2} applies  to construct a coherent $MC$-cycle from a  flat-approaching $MC$-cycle, which is canonically defined up to isotopies.

Formally we may write
$$ [Z^{( \overset{\bigtriangleup}{\mathbf{w}}, \nu, U^{\blacktriangle}  ) }] = \frac{1}{\hbar^{|Comp|}} [Z^{(\mathbf{w}, \nu, U^{\blacktriangle}) }] .$$
where $ [Z^{(\mathbf{w},\nu,  U^{\blacktriangle}) }]  \in  \in  PSH(M |    \mathbf{w},U^{\blacktriangle}) $ is the $PSH$-class defined in \cite{PSPCS2}.


\subsubsection{Twist of Decorated Multi-loops}

Let $\overset{\bigtriangleup}{\mathbf{w}}= (Comp, (V_c, \beta_c)_c , \mathbf{w} )$ be a decorated multi-loop. 
Let $(v_1,t_1), (v_2,t_2) \in V \times \mathbb{S}^1$   with $ (v_1,t_1) \neq (v_2,t_2)$ and $w_{v_1}(t_1)= w_{v_2}(t_2)$. Define the decorated multi-loop $ \overset{\bigtriangleup}{\mathbf{w}'}= (Comp', (V_c,' \beta_c')_c , \mathbf{w}' )$ as follows: 

\begin{itemize}
  \item if $v_1 \neq v_2$,  $ V'$ is obtained replacing $v_1$ and $v_2$ with an element $v_0$. $w'_{v_0}$ is obtained gluing $\mathbb{S}^1_{v_1}$ with   $\mathbb{S}^1_{v_2}$ at the points $t_1$ and $t_2$  in the only orientation preserving way ;

We have the sub-cases:
\begin{itemize}
\item $v_1 \in V_{c_1}$, $v_2 \in V_{c_2}$ for  $c_1 , c_2 \in Comp$ with $c_1 \neq c_2$: $Comp'$ is obtained from $Comp$ replacing the components $c_1$ and $c_2$ with a unique component $c_0$, $V_{c_0}$ is obtained from $V_{c_1} \sqcup V_{c_2} $ replacing the vertices $v_1$ and $v_2$ with $v_0$, 
$\beta_{c_0}=\beta_{c_1} + \beta_{c_2}$.

\item $v_1,v_2 \in V_{c_0}$ for some $c_0  \in Comp$, $v_1 \neq v_2$:  $Comp =Comp'$,
$V_{c_0}'$ is obtained from $V_{c_0} $ replacing the vertices $v_1$ and $v_2$ with $v_0$, $\beta_{c_0}'=\beta_{c_0}$.
\end{itemize}
\item $v_1=v_2  \in V_{c_0}$ for some  $c_0 \in Comp$: $  V'$ is obtained replacing $v_1$ with two elements $v_1', v_2'$. $w'_{v_1'}$ and $w'_{v_2'}$ are obtained   attaching $t_1$ and $t_2$ in $\mathbb{S}^1_{v_1}$ and splitting it in two circles in the only orientation preserving way.  $Comp' = Comp$, 
 $V_{c_0}'$ is obtained from $V_{c_0} $  replacing $v_1$ with $v_1', v_2'$,    $\beta_{c_0}'=\beta_{c_0}$.

\end{itemize}

We say that the decorated multi-loops  $ \overset{\bigtriangleup}{\mathbf{w}}$ and $ \overset{\bigtriangleup}{\mathbf{w}'}$ are related by twist.

\subsection{Multi-Curve-Skein}

The usual definition of Skein involves oriented embedded links equipped  with an isotopy class of frames of the link. In the context of  \cite{PSPCS2} it is introduced an extra geometric data called Euler Alignment and it is considered an extended version of the Skein relations. 
Since the definition of $MCH$-classes (\ref{coherent-class2}) is strictly related to the definition of $PSH$-classes of  \cite{PSPCS2}  it is immediate to translate the results of  \cite{PSPCS2} to the context of this paper. 


Recall that an Euler Alignment  $\kappa$ is defined by an isotopy between $T \mathbf{w} $ and $U |_{\mathbf{w}}$ as sections of $T M |_{\mathbf{w}}$.
To each Euler Alignment  $\kappa$ it is associated an Euler Structure  $U^{\kappa}$ compatible with $\nu$ with the property that
there is a prefered  isotopy $\tilde{U}^{\kappa}$  between   $U^{\blacktriangle}$ and $U^{\kappa}$ canonically determinate up to isotopy of isotopies.






Analogously to \cite{PSPCS2} we define
\begin{equation} \label{coherent-class2}
[Z^{(\overset{\bigtriangleup}{\mathbf{w}}, \nu, \kappa ) }] = [transfer_{\tilde{U}^{\kappa}}(Z^{ \overset{\bigtriangleup}{\mathbf{w}}, \nu,  U^{\kappa}   })] \in MCH^{\intercal}(M| \mathbf{w},  U^{\blacktriangle} ).
\end{equation}
Here, as in \cite{PSPCS2} , $\tilde{U}^{\kappa}$ is the small isotopy between $U^{\kappa}$ and $U$ canonically defined up to isotopies of isotopies. The transfer map
$transfer_{\tilde{U}^{\kappa}}$ for $MCH$ is defined as in the case of $PSH$ (see \cite{PSPCS2}).

In \cite{PSPCS2} it is showed that the coherent-$PSH$-classes satisfies an extended version of the Skein relations.
The version of the Skein arising in Point Splitting Perturbative Chern Simons  are formal power series with coefficients triples $(\mathbf{w}, \nu, \kappa)$. In  \cite{PSPCS2} we show that $PSH$ is isomorphic to a suitable completion of the (extended) Skein introduced there.
Below we state the analogous equivalence between $MCH$ and Skein. The proof are the same as in  \cite{PSPCS2}. Note that according the general philosophy of Floer theory, we may consider objects up to small isotopy. However we consider directly the completion of the Skein which allow to obtain a strict isomorphism with the normalized-$MCH$.  

We consider  arrays
\begin{equation} \label{framed-smoothing}
    ( \overset{\bigtriangleup}{\mathbf{w}^{\ast}} , \nu^{\ast},  \kappa^{\ast},  \overset{\bigtriangleup}{\mathfrak{w}})   
\end{equation}
where
$\nu^{\ast}$ is frame of $\mathbf{w}^{\ast}$, $ \kappa^{\ast}$ is an Euler Alignment of $\mathbf{w}^{\ast}$, $ \overset{\bigtriangleup}{\mathfrak{w}}$ is a finite set of decorated multi-loops closed by twist relation.
We assume that the decorated oriented link $ \overset{\bigtriangleup}{\mathbf{w}^{\ast}}$ is close in the $C^0$-topology to an element  $ \overset{\bigtriangleup}{\mathbf{w}}$ of   $  \overset{\bigtriangleup}{\mathfrak{w}}$ .   In the case $ \mathfrak{w} = \emptyset$  we assume $\mathbf{w}^{\ast}$ is close in the $C^0$-topology to a constant multi-loop.

We consider two arrays (\ref{framed-smoothing}) equivalent
\begin{equation} 
\label{isotopy-skein-equivalence-singular}
(  \overset{\bigtriangleup}{\mathbf{w}^{\ast}}, \nu^{\ast},  \kappa^{\ast},   \overset{\bigtriangleup}{\mathfrak{w}}) \sim   ( \mathbf{w}^{\ast, \prime}, \nu^{\ast, \prime},  \kappa^{\ast, \prime},  \overset{\bigtriangleup}{\mathfrak{w}})  
 \end{equation}
if $ \overset{\bigtriangleup}{\mathbf{w}^{\ast}}$ is isotopic to $ \overset{\bigtriangleup}{\mathbf{w}^{\ast, \prime}}$ with isotopy close in the $C^0$-topology to an element  $ \overset{\bigtriangleup}{\mathbf{w}}$ of   $  \overset{\bigtriangleup}{\mathfrak{w}}$, $\nu^{\ast}$ is isotopic to $ \nu^{\ast, \prime}$ and $ \kappa^{\ast, \prime}$ is isotopic to $ \kappa^{\ast}$.


Define the singular Multi-Curve-Skein $ MC-\overline{Skein}(M)[[\hbar,a]]^+$ as  the space of formal Laurent series with coefficients in the arrays
(\ref{framed-smoothing}) 
$$\sum_s \hbar^{k_s} a^{d_s}  (\overset{\bigtriangleup}{\mathbf{w}^{\ast}}_s, \nu^{\ast}_s,  \kappa^{\ast}_s , \overset{\bigtriangleup}{\mathfrak{w}}_s) $$
with $d_s \geq 0 , k_s \geq -|Comp_s|$,
modulo   the following  equivalence relations (where we make an obvious adaptation of the notations of  \cite{PSPCS2}): 
\begin{itemize}
\item
\begin{equation} \label{crossing-skein-equivalence-singular}
  \beta (   \overset{\bigtriangleup}{\mathbf{w}^{\ast}}_+,  \nu_+^{\ast}, \kappa^{\ast}_+, \overset{\bigtriangleup}{\mathfrak{w}})-  \beta^{-1} (  \overset{\bigtriangleup}{\mathbf{w}^{\ast}}_- , , \nu_-^{\ast}, \kappa^{\ast}_- ,  \overset{\bigtriangleup}{\mathfrak{w}}) \sim A ( \overset{\bigtriangleup}{\mathbf{w}^{\ast}}_0 , \nu_0^{\ast}, \kappa^{ \ast}_0,   \overset{\bigtriangleup}{\mathfrak{w}});
 \end{equation}
  \item
 \begin{equation} \label{twisting-skein-equivalence-singular}
 (  \overset{\bigtriangleup}{\mathbf{w}^{\ast}}, \nu^{\ast, +1}, \kappa^{\ast},  \overset{\bigtriangleup}{\mathfrak{w}}) \sim \alpha(\hbar,a) (  \overset{\bigtriangleup}{\mathbf{w}^{\ast}}, \nu^{\ast}, \kappa^{ \ast},  \overset{\bigtriangleup}{\mathfrak{w}} );
 \end{equation} 
\item
\begin{equation} \label{shift-euler-skein-equivalence-singular}
  (  \overset{\bigtriangleup}{\mathbf{w}^{\ast}}, \nu^{\ast}, \kappa^{ \ast, +1} ,   \overset{\bigtriangleup}{\mathfrak{w}}) \sim  \theta(\hbar ,a ) (  \overset{\bigtriangleup}{\mathbf{w}^{\ast}}, \nu^{\ast}, \kappa^{\ast},    \overset{\bigtriangleup}{\mathfrak{w}});
  \end{equation}  
  \item
 \begin{equation} \label{unknot-skein-equivalence-singular}
  (   \overset{\bigtriangleup}{\mathbf{w}^{\ast}} \sqcup \mathbf{w}_{unknot}, \nu^{\ast} \sqcup \nu_{unknot} , \kappa^{\ast} \sqcup \kappa_{unknot} ,  \overset{\bigtriangleup}{\mathfrak{w}}) \sim   r(\hbar, a)(   \overset{\bigtriangleup}{\mathbf{w}^{\ast}}, \nu^{\ast}  , \kappa^{\ast},   \overset{\bigtriangleup}{\mathfrak{w}} ) .
  \end{equation}
 \end{itemize}
The formal power series appearing in the relations above are the same formal power series appearing in \cite{PSPCS2}.

Isotopies of $ MC-\overline{Skein}(M)[[\hbar,a]]^+$ are defined straightforwardly considering deformations of $\overset{\bigtriangleup}{\mathfrak{w}}$ in (\ref{framed-smoothing}).

The construction of \cite{PSPCS2} yields the following
\begin{proposition}
There exists a natural map of $\Q[[\hbar,a]]$-moduli
\begin{equation} \label{singular-skein-map}
MC-\overline{Skein}(M)[[\hbar,a]]^+ \rightarrow MCH^{\intercal}(M| U^{\blacktriangle}).
\end{equation}
The map (\ref{singular-skein-map}) is an isomorphism  compatible with isotopies.
\end{proposition}

The map (\ref{singular-skein-map}) is defined analogously to  \cite{PSPCS2}
\begin{equation} \label{skein-map-singular}
 ( \overset{\bigtriangleup}{\mathbf{w}^{\ast}}, \nu^{\ast},  \kappa^{\ast}, \overset{\bigtriangleup}{\mathfrak{w}}) \mapsto transfer_{\tilde{\mathbf{w}}^{\ast}} (  Z^{(  \overset{\bigtriangleup}{\mathbf{w}^{\ast}} , \nu^{\ast}, \kappa^{\ast} )}) 
 \end{equation}
where $ \overset{\bigtriangleup}{\tilde{\mathbf{w}}^{\ast}}$ is a small isotopy between  $ \overset{\bigtriangleup}{\mathbf{w}^{\ast}}$ and an element  $ \overset{\bigtriangleup}{\mathbf{w}}$ of   $  \overset{\bigtriangleup}{\mathfrak{w}}$.

\end{document}